 \documentclass[runningheads]{llncs}
\usepackage[T1]{fontenc}
\usepackage{amsmath,amssymb}
\usepackage{mathptmx}
\usepackage{graphicx,graphics}
\usepackage{tikz}
\usetikzlibrary{decorations.markings}
\usepackage{booktabs} 
\usepackage{xcolor}   
\usepackage{multirow}   
\usepackage{array}      
\usepackage{caption} 
\usepackage[misc,geometry]{ifsym}
\begin{document}
\title{On the Word-Representability of 5-Regular Circulant Graphs}
\titlerunning{On the Word-Representability of 5-Regular Circulant Graphs}
%
\author{Suchanda Roy\thanks{A preliminary version of this work has been accepted for presentation at CALDAM~2026.}\and
Ramesh Hariharasubramanian \Letter}
\authorrunning{S. Roy et al.}
%
\institute{Department of Mathematics, Indian Institute of Technology Guwahati, \\Guwahati, Assam 781039, India\\
\email{r.suchanda@iitg.ac.in}\\
\email{\Letter~ramesh\_h@iitg.ac.in}}
 \maketitle              
\begin{abstract}
A graph $G = (V, E)$ is \emph{word-representable} if there exists a word $w$ over the alphabet $V$ such that, for any two distinct vertices $x, y \in V$, $xy \in E$ if and only if $x$ and $y$ alternate in $w$.  
Two letters $x$ and $y$ are said to \emph{alternate} in $w$ if, after removing all other letters from $w$, the resulting word is of the form $xyxy\dots$ or $yxyx\dots$ (of even or odd length). For a given set \(R = \{r_1, r_2, \dots, r_k\}\) of jump elements, an undirected \emph{circulant graph} \(C_n(R)\) on \(n\) vertices has vertex set 
\(\{0, 1, \dots, n-1\}\) and edge set
$
E = \left\{ \{i,j\} \;\middle|\; |i - j| \bmod n \in \{r_1, r_2, \dots, r_k\} \right\},
$
where \(0 < r_1 < r_2 < \dots < r_k < \frac{n}{2}\).
Recently, Kitaev and Pyatkin proved that every \(4\)-regular circulant graph is word-representable. 
Srinivasan and Hariharasubramanian further investigated circulant graphs and obtained bounds on the representation number for \(k\)-regular circulant graphs with \(2 \le k \le 4\).  
In addition to these positive results, their work also presents examples of non-word-representable circulant graphs. In this work, we study word-representability and the representation number of 5-regular circulant graphs via techniques from elementary number theory and group theory, as well as graph coloring, graph factorization and morphisms.
\end{abstract}
\section{Introduction}
The theory of \emph{word-representable graphs} is a rich and promising area of research, with strong connections to algebra, graph theory, combinatorics on words, formal languages, and scheduling problems. This concept was first introduced by Kitaev and Pyatkin in \cite{kitaev2008representable}, based on the study of Kitaev and Seif on the celebrated \emph{Perkins semigroup} \cite{kitaev2008word}. Word-representable graphs generalize several fundamental graph classes, including \emph{circle graphs, comparability graphs, and 3-colorable graphs}. We refer the reader to \cite{kitaev2015words},\cite{kitaev2017comprehensive} for motivation and a detailed discussion of the connections between word-representable graphs and various fields. Notably, determining whether a graph is word-representable or not is a \emph{NP-complete} problem.

The notion of \emph{semi-transitive orientation} (Definition \ref{def}), introduced by Halldórsson in~\cite{halldorsson2016semi}, plays a central role in characterizing \emph{word-representable graphs}. In fact, a graph is word-representable if and only if it admits such an orientation (Theorem~\ref{semi}).

 Circulant graph is an important family in graph theory due to their high degree of symmetry, as they are regular and vertex-transitive. These properties, together with their connections to number theory and applications in areas such as network design, coding theory, and cryptography, make them a natural object of study. These graphs have been extensively investigated with respect to their structural properties, including connectivity, planarity, and factorization. In recent years, word-representability of circulant graphs has attracted attention. In \cite{kitaev2020semi}, Kitaev and Pyatkin showed that every \(4\)-regular circulant graph is 
word-representable, while Srinivasan and Hariharasubramanian examined word-representability of circulant graphs and established bounds on the representation number for \(k\)-regular circulant 
graphs with \(2 \leq k \leq 4\) in \cite{srinivasan2025semi}. Their work also highlights that not all circulant graphs are word-representable, as explicit 
counterexamples are known. The coexistence of such examples motivates us to extend the study to \(5\)-regular circulant graphs.

 This paper presents several contributions to the study of word-representability of 5-regular circulant graphs. We begin by dividing this class into subclasses according to the parity of the jump set elements and analyze their word-representability (Theorem~\ref{odd}, Remark~\ref{even}). Also, we give particular attention to the subclass of 5-regular circulant graphs in which at least one jump set element is relatively prime to the number of vertices. Their word-representability is established in 
Theorems~\ref{nx2}, \ref{nx0}, \ref{n0x1}, \ref{n0x2}, \ref{n1x2}, \ref{n2x1}, as well as in   
Corollaries \ref{cnx1}, \ref{cnxo}, \ref{cnox1}, \ref{cn2x1}, \ref{cnox2}  and 
\ref{cn1x2}.  
Moreover, upper bounds on their representation numbers are provided in 
Theorems~\ref{20}, \ref{21}, and \ref{12n}. We then extend the analysis to the general case of 5-regular circulant graphs (Theorem~\ref{19}, \ref{17},    \ref{6} and \ref{7}). The paper ends with several open problems and directions for future research.

This paper is organized as follows. Section \ref{2} reviews basic results from elementary number theory and group theory, provides the necessary background on word-representability in \ref{2.1}, and introduces circulant graphs along with their key properties in \ref{2.2}.
 In Section \ref{result}, we present our results on the word-representability of 5-regular circulant graphs and give upper bounds on their representation number. Finally, Section \ref{con} concludes the paper with open problems and directions for future research.

\section{Preliminaries}\label{2}
Before presenting the necessary background on word-representability, we first state some basic results and definitions from elementary number theory and group theory that will be used in this work.
\begin{theorem}\label{num}
A linear congruence of the form 
\(
b x \equiv a \pmod{n}
\) 
has a unique solution modulo \(n\) if and only if \(\gcd(b, n) = 1\).

\end{theorem}
\begin{definition}
Let \(H\) be a group and \(a \in H\). The \emph{order} of \(a\), denoted \(o(a)\), is the smallest positive integer \(m\) such that 
$
a^m = e,
$ 
where \(e\) is the identity element of \(H\). If no such \(m\) exists, \(a\) is said to have infinite order.
\end{definition}

\begin{proposition}
Let \(a \in H\) be an element of finite order \(o(a)\), and let \(x\) be a positive integer. Then the order of \(a^x\) is given by
$
o(a^x) = \frac{o(a)}{\gcd(o(a), x)}.
$
\end{proposition}
\begin{definition}
A group \(H\) is called \emph{cyclic} if there exists an element \(g \in H\) such that every element of \(H\) can be written as a power of \(g\); that is,
$
H = \{ g^k \mid k \in \mathbb{Z} \}.
$
In this case, \(g\) is called a \emph{generator} of \(H\).
\end{definition}

\begin{proposition}
Let \(H = \langle g \rangle\) be a cyclic group of order \(n\). An element \(g^k \in H\) is a generator of \(H\) if and only if
$
\gcd(k, n) = 1.
$
\end{proposition}
 We have applied these results to the cyclic group \((\mathbb{Z}_{2n},+)\) in Theorems~\ref{15} and~\ref{17}. As \((\mathbb{Z}_{2n},+)\) is an additive group, instead of $g^k$ we write $k.g$, where $g$ is a generator of \((\mathbb{Z}_{2n},+)\). In the following, we present a brief overview of word-representability.

\subsection{Word-Representable Graphs}\label{2.1}
Throughout this paper, we assume that all graphs are finite, simple, and undirected. For a graph \(G\), we denote its vertex set by \(V(G)\) and its edge set by \(E(G)\). For vertices \(x, y \in V(G)\), we write \(x \sim y\) if there is an edge between \(x\) and \(y\), and \(x \nsim y\) if there is no edge between them.
We begin this section with the definition of word-representable graphs.

\begin{definition}[\cite{kitaev2015words}]  A simple graph $G = (V, E)$ is \textit{word-representable} if there exists a word $w$ over the alphabet $V$ such that letters $x$ and $y$, $\{x,y\}\in V$ alternate in $w$ if and only if $xy \in E$, i.e., $x$ and $y$ are adjacent for each $x\not=y$. If a word $w$ \textit{represents} $G$, then $w$ contains each letter of $V(G)$ at least once.
\end{definition}
\begin{example}
    For instance, the word $1234$ represents the complete graph $K_4$ on four vertices. More generally, any permutation of $n$ distinct symbols represents the complete graph $K_n$ on $n$ vertices. Also, $162132435465$ represents $C_6$, the cycle graph on 6 vertices.

\end{example}
\begin{proposition}[\cite{kitaev2015words}]\label{pr1}
 		Let $w = uv$ be a $k$-uniform word representing a graph $G$, where $u$ and $v$ are two, possibly empty, words. Then, the word $w' = vu$ also represents $G$.
\end{proposition}
 \begin{definition}[\cite{kitaev2015words}]
A word $w$ contains a word $u$ as a \textit{factor} if $w = xuy$ where $x$ and $y$ can be empty words. 
\end{definition}
\begin{definition}[\cite{kitaev2015words}]
A \textit{subword} of a word $w$ is a word obtained by removing certain letters from $w$. In a word $w$, if $x$ and $y$ alternate, then $w$ contains $xyxy\cdots$ or $yxyx\cdots$ (odd or even length) as a subword.
\end{definition}
\begin{example}
    The word $1241154325$ contains the words $154$ and $24$ as factors, while $24142$ is a subword of it.
\end{example}
\begin{definition}
     [\cite{kitaev2015words}] \textit{$k$-uniform word} A word $w$ is said to be $k$-uniform word if every letter occurs exactly $k$ times in it.
 \end{definition}

 \begin{definition}[\cite{kitaev2015words}] A graph is said to be $k$-word-representable if there exists a $k$-uniform word representing it.
 \end{definition}

 \begin{definition} [\cite{kitaev2017comprehensive}]
     For a word-representable graph $G$, \textit{the representation number} is the least $k$ such that $G$ is $k$-representable and it is denoted by $R(G)$.
 \end{definition}
 The following Theorem gives the representation number of cycle graph.
 \begin{theorem}[\cite{kitaev2015words}]\label{cycle}
Let \(G\) be a cycle graph. Then the representation number of \(G\) is 2.
\end{theorem}
In the following, we state one of the key developments in the study of word-representable graphs, namely its characterization in terms of semi-transitive orientations, which is defined based on shortcuts in \cite{halldorsson2016semi}.
\begin{definition}[\cite{kitaev2015words}]
    A \emph{semi-cycle} is the directed acyclic graph obtained by reversing the direction of one edge of a directed cycle.
\end{definition}
\begin{definition}[\cite{halldorsson2016semi}]\label{def}
    An acyclic digraph is called a \emph{shortcut} if it is induced by the vertices of a semi-cycle and contains a pair of non-adjacent vertices. In particular, any shortcut  

\begin{itemize}
    \item is acyclic (that is, it has no directed cycles);
    \item contains at least four vertices;
    \item has exactly one source (a vertex with no incoming edges) and exactly one sink (a vertex with no outgoing edges), with a directed path from the source to the sink passing through every vertex;
    \item includes an edge connecting the source to the sink, called the \emph{shortcutting edge};
    \item is not transitive (that is, there exist vertices $u$, $v$, and $z$ such that $u \to v$ and $v \to z$ are edges, but $u \to z$ is not).
\end{itemize}

\begin{definition}[\cite{halldorsson2016semi}]
An orientation of a graph is \emph{semi-transitive} if it is acyclic and shortcut-free.
\end{definition}
 
\end{definition}
\begin{theorem}[\cite{halldorsson2016semi}]\label{semi}
A graph \(G\) is word-representable if and only if it admits a semi-transitive orientation.
\end{theorem}
\begin{remark}\label{bip}
The notion of semi-transitive orientation generalizes the classical concept of transitive orientation, where an orientation is \emph{transitive} if the presence of \(u \rightarrow v\) and \(v \rightarrow z\) implies \(u \rightarrow z\). Graphs admitting such an orientation are called \emph{comparability graphs}. In particular, path graphs, bipartite graphs and complete graphs admit transitive orientations, and therefore they are word-representable.
\end{remark}
\begin{theorem}[\cite{kitaev2015words}]\label{3-col}
 If a graph $G$ is 3-colorable, then it is word-representable.
\end{theorem}
The above theorem establishes a connection between colorability and word-representa-\\bility, which plays an important role in the proofs of Theorems~\ref{19} and~\ref{17}. Furthermore, the following theorem allows us to restrict the study of word-representable graphs to connected graphs.

\begin{theorem}[\cite{kitaev2015words}]\label{conn}
Let \(G\) be a graph with connected components \(G_1, G_2, \dots, G_k\). Then \(G\) is word-representable if and only if each \(G_i\) is word-representable. Moreover, the representation number of \(G\) satisfies
\(
R(G) = \max\{R(G_1), R(G_2), \dots, R(G_k)\}.
\)
\end{theorem}
In the following, we present results on the word-representability of Cartesian products of graphs. We begin with the definition of the Cartesian product.

\begin{definition}
Let $G_{1} = (V(G_{1}), E(G_{1}))$ and $G_{2} = (V(G_{2}), E(G_{2}))$ be two graphs.  
The \emph{Cartesian product} $G_{1} \square G_{2}$ is the graph defined as follows:
\begin{itemize}
    \item The vertex set is $V(G_{1} \square G_{2}) = V(G_{1}) \times V(G_{2})$.
    \item Two vertices $(u_{1}, v_{1})$ and $(u_{2}, v_{2})$ are adjacent in $G_{1} \square G_{2}$ if and only if either
    \begin{enumerate}
        \item $u_{1} = u_{2}$ and $v_{1}$ is adjacent to $v_{2}$ in $G_{2}$, or
        \item $v_{1} = v_{2}$ and $u_{1}$ is adjacent to $u_{2}$ in $G_{1}$.
    \end{enumerate}
\end{itemize}
\end{definition}

\begin{theorem}[\cite{kitaev2015words}]\label{carts}
If $G_{1}$ and $G_{2}$ are word-representable graphs, then their Cartesian product $G_{1} \square G_{2}$ is also word-representable.
\end{theorem}
\begin{theorem}[\cite{srinivasan2024minimum}]\label{rcarts}
Let $G_{1}$ and $G_{2}$ be two word-representable graphs with their representation numbers $r_{1}$ and $r_{2}$, respectively.  
Then the Cartesian product $G_{1} \square G_{2}$ is $(r_{1} + r_{2} + \min\{|G_{1}|, |G_{2}|\})$-representable, i.e, $R(G)\leq (r_{1} + r_{2} + \min\{|G_{1}|, |G_{2}|\}) $  
\end{theorem}


 
\subsection{Circulant Graph}\label{2.2}
A circulant graph \(C_n(R)\) for a set \(R = \{r_{1}, r_{2}, \ldots, r_{k}\}\) is defined as the graph with vertex set 
\(V(G)=\{0,1,\ldots,n-1\}\) and edge set 
\(E(G)=
\{\, ij \mid |i-j| \pmod{n} \in \{r_{1}, r_{2}, \ldots, r_{k}\}\,\},
\)
where \(0 < r_{1} < r_{2} < \cdots < r_{k} < \tfrac{n+1}{2}\). $R$ is known as the jump set and $r_{1}, r_{2}, \cdots r_{k}$ are known as jump elements of the circulant graph.  

The graph \(C_n( r_{1}, r_{2}, \ldots, r_{k})\) is regular of degree
\(
d = 
\begin{cases}
2k, & \text{if } r_{k} \neq \tfrac{n}{2}, \\[6pt]
2k-1, & \text{if } r_{k} = \tfrac{n}{2}.
\end{cases}
\)


In the following, we will discuss some fundamental properties of circulant graphs.
\begin{theorem}[\cite{vilfred2013theory}]\label{thm:cycle-period}
Let \(C_{n}(R)\) be a circulant graph and let \(r \in R\). 
Then a cycle of period \(r\) in \(C_{n}(R)\) has length \(\tfrac{n}{\gcd(n,r)}\), 
and there are exactly \(\gcd(n,r)\) vertex-disjoint periodic cycles of period \(r\).
\end{theorem}
\begin{theorem}[\cite{vilfred1994sigma}, \cite{boesch1984circulants}]\label{thm:connected-circulant}
A circulant graph \(C_n( r_{1}, r_{2}, \ldots, r_{k})\) is connected 
if and only if 
\(
\gcd(n, r_{1}, r_{2}, \ldots, r_{k}) = 1.
\) 
Moreover, if \(d = \gcd(n, r_{1}, r_{2}, \ldots, r_{k})\), then
\(
C_n( r_{1}, r_{2}, \ldots, r_{k}) \cong 
d \cdot C_{\tfrac{n}{d}}\!\left(\tfrac{r_{1}}{d}, \tfrac{r_{2}}{d}, \ldots, \tfrac{r_{k}}{d}\right),
\)
i.e, \(C(n; r_{1}, \ldots, r_{k})\) is isomorphic to the disjoint union of \(d\) copies of 
\(C\!\left(\tfrac{n}{d}; \tfrac{r_{1}}{d}, \ldots, \tfrac{r_{k}}{d}\right)\).
\end{theorem}
\begin{theorem}[\cite{heuberger2003planarity}]\label{thm:bipartite-circulant}
Let \(G := C_{n}(r_{1}, r_{2}, \ldots, r_{k})\) be a connected circulant graph.  
Then \(G\) is bipartite if and only if \(n\) is even and each \(r_{i}\) is odd for \(1 \leq i \leq k\).
\end{theorem}
The following two theorems establish that circulant graphs can be factorized uniquely into a Cartesian product of prime graphs. 
\begin{theorem}[\cite{vilfred2013theory}]\label{th7} 
For $n \in \mathbb{N}$ and a set $R = \{r_{1}, r_{2}, \ldots, r_{k}\}$, the Cartesian product 
\(
P_{2} \,\square\, \\
C(2n+1; R) \;\cong\; C\bigl(2(2n+1);\, 2R \cup \{2n+1\}\bigr) 
\;\cong\; C\bigl(2(2n+1);\, 2dR \cup \{2n+1\}\bigr),
\)
where $\gcd\!\bigl(2(2n+1), d\bigr) = 1$.
\end{theorem}
\begin{theorem}[\cite{vilfred2013theory} Factorization Theorem of Circulant Graphs]\label{fac}
Let $p$ and $q$ be relatively prime integers. If $R \subseteq [1, p/2]$,  
$S \subseteq [1, q/2]$, and $T \subseteq [1, pq/2]$ with  
\(
T = d q R \,\cup\, d p S
\)
for some $d$ such that $\gcd(pq, d) = 1$, then
\(
C_{pq}(T) \;\cong\; C_{p}(R) \,\square\, C_{q}(S)
\).
\end{theorem}
We now present the known results regarding the word-representability and non-word-representability of circulant graphs.

\begin{theorem}[\cite{srinivasan2025semi}]\label{3reg}
Let \(G \cong C_{2n}( a, n)\) be a \(3\)-regular circulant graph. Then it is word-representable and \(R(G) \leq 3\).
\end{theorem}
\begin{theorem}[\cite{kitaev2020semi}]
Let \(G \cong C_n( a, b)\) be a \(4\)-regular circulant graph. Then \(G\) is word-representable.
\end{theorem}
\begin{theorem}[\cite{srinivasan2025semi}]
  Let \(G \cong C_n( 1, a)\) be a connected \(4\)-regular circulant graph with 
\(\frac{n}{3} \leq a < \frac{n}{2}\) and \(n > 6\).  
Then $G$ is word-representable and \(R(G) \leq 4\).
  
\end{theorem}
\begin{remark}\label{2reg}
     As we know that any 2-regular connected graph is a cycle, then any 2-regular connected circulant graph is also a cycle. Hence, by Theorem \ref{cycle}, $G$ is word-representable and $R(G)=2$ where $G$ is any 2-regular circulant graph.

\end{remark}
The following theorem will provide examples of non-word-representable circulant graphs.
\begin{theorem}[\cite{srinivasan2025semi}]\label{thm:new}
Let $G=C_n(r,r+1,\ldots,2r)$. Then $G$ is not word-representable whenever $2<\tfrac{n+1}{5} < r < \tfrac{n-1}{4}$.
\end{theorem}
Here, we presented some results concerning the word-representability of $k$-regular circulant graphs for $k \leq 4$, and provided explicit examples of circulant graphs that are not word-representable. In the following section  we pursue the investigation of 5-regular circulant graphs.

\noindent 
\section{Word-representability of 5-regular Circulant Graph}\label{result}
In this section, we first divide the class of 5-regular circulant graphs into subclasses based on the parity of the jump set elements and analyze their word-representability (Theorem \ref{odd}, Remark \ref{even}). We then study word-representability more generally and, in certain cases, establish upper bounds on the representation number by employing techniques such as colorability (\ref{col}), morphisms (\ref{mor}) and factorization into smaller circulant graphs (\ref{factor}).  
  
\begin{theorem}\label{odd}
Let $G = C_{2n}(a,b,n)$ be a $5$-regular circulant graph, where $n$ is odd. 
If $a$ and $b$ have the same parity, i.e., both are even or both are odd, then $G$ is word-representable.

\begin{proof}
   Let $G=C_{2n}(a,b,n)$ be a 5-regular circulant graph, where $n$ is odd. We consider two cases as below.
   \begin{itemize}
       \item  Both $a$ and $b$ are even. In this case, using Theorem \ref{th7} we get $G=C_{2n}(a,b,n) \cong P_2 \square C_n(\frac{a}{2}, \frac{b}{2})$, where $P_2$ is the path of $2$ vertices. Now, combining Remark \ref{bip} and Theorem \ref{carts}, we conclude that $G$ is word-representable. 
       \item Both $a$ and $b$ are odd. In this case, using Theorem \ref{thm:bipartite-circulant}, we conclude that $G$ is bipartite. Therefore, $G$ is word-representable by Remark \ref{bip}.
   \end{itemize}
\end{proof}
\end{theorem}
\begin{remark}\label{even}
Let $G = C_{2n}(a,b,n)$ be a $5$-regular circulant graph where $n$, $a$, and $b$ are even. As $gcd(a,b,n)\geq2$, then by Theorem \ref{thm:connected-circulant}, $G$ is not connected. Hence, $G$ is word-representable if and only if all of its connected components are word-representable (by Theorem~\ref{conn}). From this point onward, we restrict our attention to connected $5$-regular circulant graphs.  
 
\end{remark}

\begin{theorem}\label{15}
Let $G=C_{2n}(a,b,n)$ be a circulant graph with $ gcd(b,2n)=1 $. Then $G \cong C_{2n}(x,1,n)$
, where  $bx \equiv a \pmod{2n}$. 
\begin{proof}
Since $\gcd(b,2n)=1$, the element $b$ is a generator of the cyclic group $\mathbb{Z}_{2n}$.  
Hence, the vertices of the circulant graph can be arranged on a cycle as  
\(
\{0,\, 1\cdot b,\, 2\cdot b,\, \ldots,\,(n-1)\cdot b\}.
\) 
Consider the isomorphism
\(
f : V\big(C_{2n}(a,b,n)\big) \;\longrightarrow\; V\big(C_{2n}(x,1,n)\big)\}
\)
defined by
\(
f(i\cdot b) = i,  \forall\, i \in \{0,1,2,\ldots,n-1\}.
\)

\begin{itemize}
    \item If $\lvert i-j \rvert = 1$, then $i \cdot b \sim j \cdot b \iff i \sim j$.
    
    \item If $\lvert i-j \rvert > 1$, then 
    \(
    i \cdot b \sim j \cdot b \implies  i \cdot b - j \cdot b \equiv \pm a \pmod{2n} 
    \;\;\text{or}\;\;
    i \cdot b - j \cdot b \equiv n \pmod{2n}.
    \)\\
    Now, $i \cdot b-j \cdot b \equiv \pm a \pmod{2n} \iff i-j \equiv \pm x \pmod {2n}$ by Theorem \ref{num}. Hence $i \cdot b \sim j \cdot b \iff i  \sim j.$\\
    Since $2n\equiv 0 \pmod{2n}$, $o(n)=2$. Let $n=k.b$. Then $o(k.b)=\frac{o(b)}{gcd(k,o(b))}$, which implies $\frac{2n}{gcd(k,2n)}=2$, i.e, $k=n$. Hence $n.b=n$. Then, $i \cdot b - j \cdot b \equiv n \pmod{2n} \implies (i-j) \equiv n \pmod {2n}$ by Theorem \ref{num}. Again, $i-j\equiv n\pmod{2n}\implies i.b-j.b\equiv n.b\equiv n\pmod{2n}$. Hence $i \cdot b \sim j \cdot b \iff i  \sim j$.
\end{itemize}
Hence, $C_{2n}(a,b,n) \cong C_{2n}(x,1,n)$.
\end{proof}
\end{theorem}
In the following sections, for the class of 5-regular circulant graphs with either $\gcd(a,2n)\\
=1$ or $\gcd(b,2n)=1$, we restrict our attention to the graphs of the form $C_{2n}(x,1,n)$.

\subsection{Word-Representability and Colorability of 5-Regular Circulant Graphs}\label{col}

In this section, we establish the word-representability of \(5\)-regular circulant graphs by relating it to their colorability. 
Our approach relies on Theorem~\ref{3-col} together with the embedding technique described in \cite{huang2024embedding}. Basically, for a given graph \(G\), we find a word-representable graph \(H\) with a homomorphism 
\(f: G \to H\), orient \(G\) according to a fixed acyclic orientation of \(H\), and then check \(H\) 
for shortcutting edges; if none exist, we conclude that the induced orientation of \(G\) is semi-transitive and if there exists some shortcutting edge, we check whether it is inducing a shortcut in $G$ or not. If no shortcut is induced, then we conclude that the induced orientation in $G$ is semitransitive.

\begin{theorem}\label{17}
Let $G=C_{2n}(a,b,n)$ be a 5-regular circulant graph. $G$ is word-represent-\
able if there exists a generator $g$ of the underlined cyclic group $(Z_{2n},+)$ such that $r=min\{p,q,2n-p,2n-q\}\geq \lceil\frac{2n}{3}\rceil$, where $a=p.g$ and $b=q.g$.
\begin{proof}
We claim that $r=min\{p,q,2n-p,2n-q\}<n$. To prove this we first prove that for any  $j\in(Z_{2n},+)$, $o(j)=2 \iff j=n.g$. Let $o(j)=2$ and $j=x.g$ for some $x\in\{0,1,2\cdots 2n-1\}$. Then $o(j)=o(x.g)\implies 2=\frac{o(g)}{gcd(x,o(g))}\implies gcd(x,2n)=n\implies x=n$. Again, let $j=n.g$. Then $o(j)=o(n.g)\implies o(j)=\frac{o(g)}{gcd(n,o(g))}\implies o(j)=2.$

Moreover, since \(o(n) = 2\), the element \(n = n.g\) is the unique element of order \(2\) in \((\mathbb{Z}_{2n}, +)\).  
As \(a \neq n\), it follows that \(p\) is either greater than \(n\) or less than \(n\).  
Without loss of generality, assume \(p > n\). Then \(2n - p < n\), which in turn implies \(r < n\).

Next, we define
$
X_0=\{j\in V(G)\mid j=k.g \hspace{.3em}\text{where} \hspace{.3em}, 0\leq k\leq r-1  \},
$ $
X_1=\{j\in V(G)\mid j=k.g \hspace{.3em}\text{where} \hspace{.3em}, r\leq k\leq 2r-1  \}
 $ and $X_2= V(G)\setminus(X_0 \cup X_1)$.
Clearly, $X_0,X_1,X_2$ form a partition of $V(G)$. We assign color $c_i$ to every vertex in $X_i$ for $i\in \{0,1,2\}$.  

Let $j_{1}\in X_i$ and $j_{1}=k_{1}.g$. Suppose, for contradiction, $j_{1} \sim j_{2}(=k_{2}.g)$ and $ j_{2} \in X_i$. Then $|k_{1}-k_{2}| \in \{p,q,2n-p,2n-q,n\}$, but $|k_{1}-k_{2}|\leq r-1$ where $r=min\{p,q,2n-p,2n-q\}$. Hence, we get a contradiction. Thus, $G$ is 3-colorable and word-representable by Theorem \ref{3-col}.

\end{proof}

\end{theorem}

In the following theorems (From Theorem~20 upto Theorem~26), we assign colors to the vertices of 
the circulant graph based on their congruence classes modulo \(3\). As the vertices are already indexed modulo \(2n\), 
Table~\ref{rule} is used to determine the corresponding congruence classes modulo 3 of the adjacent vertices 
\( (j+x)_{2n} \), \( (j+n)_{2n} \), \( (j+2n-x)_{2n} \), \( (j+1)_{2n} \) and \( (j+2n-1)_{2n} \) for any given vertex \( j \).
\vspace{-.5cm}
\begin{table}[h!]
\centering
\begin{tabular}{|c|c|c|}
\hline
\textbf{Adjencies of Vertex \textbf{$j$}} & \multicolumn{2}{c|}{\textbf{Modulo 3 Congruence Classes of the Adjacencies of $j$}} \\ \hline

\multirow{2}{*}{$(j+1)_{2n}$}
& \textbf{if $j\neq 2n-1$} & $(j+1)\pmod{3})$ \\ \cline{2-3}
& \textbf{if $j=2n-1$} & $0$ \\ \hline

\multirow{2}{*}{$(j+2n-1)_{2n}$}
& \textbf{if $j\neq 0$} & $(j-1)\pmod{3})$ \\ \cline{2-3}
& \textbf{if $j=0$} & $2n-1$ \\ \hline

\multirow{2}{*}{$(j+2n-x)_{2n}$}
& \textbf{if $j\geq x$} & $(j-x)\pmod{3})$ \\ \cline{2-3}
& \textbf{if $j< x$} & $(2n-x+j)\pmod{3})$ \\ \hline

\multirow{2}{*}{$(j+x)_{2n}$}
& \textbf{if $j\geq 2n- x$} & $(j+x-2n)\pmod{3})$ \\ \cline{2-3}
& \textbf{if $j< 2n-x$} & $(j+x)\pmod{3})$ \\ \hline

\multirow{2}{*}{$(j+n)_{2n}$}
& \textbf{if $j\geq n$} & $(j-n)\pmod{3})$ \\ \cline{2-3}
& \textbf{if $j< n$} & $(j+n)\pmod{3})$ \\ \hline

\end{tabular}
\vspace{.2cm}
\caption{Modulo 3 Congruence Classes of the Vertices Adjacent to $j$ }
\label{rule}
\end{table}
\vspace{-1cm}

\begin{theorem}\label{19}
Let $G=C_{2n}(a,b,n)$ be a 5-regular circulant graph. If $3$ does not divide any of $a, b, 2n-a, 2n-b$, or $n$, then $G$ is word-representable.
\end{theorem}

\begin{proof}
We define 
$
X_i=\{\, j\in V(G)\mid j\equiv i \pmod{3}\,\}, \quad i=0,1,2.
$
Clearly, $X_0,X_1,X_2$ form a partition of $V(G)$. We assign color $c_i$ to every vertex in $X_i$.  

Let $j\in X_i$. Suppose, for contradiction, $(j+a)_{2n}\in X_i$. Then
\(
j \equiv (j+a)_{2n} \equiv i \pmod{3} \quad\\
\implies  a \equiv 0 \pmod{3},
\)
Which contradicts the assumption $3\nmid a$. The same argument can be applied to $b, 2n-a, 2n-b$, and $n$, ensuring that no two vertices within $X_i$ are adjacent.  
Therefore, $G$ is 3-colorable and by Theorem \ref{3-col}, it follows that $G$ is word-representable.
\end{proof}

\begin{corollary}\label{cnx1}
Let $G=C_{2n}(x,1,n)$ be a 5-regular circulant graph. If $n \equiv x \equiv 1 \pmod{3}$, then $G$ is word-representable.
\end{corollary}

\begin{theorem}\label{nx2}
Let $G=C_{2n}(x,1,n)$ be a 5-regular circulant graph. If $n \equiv x \equiv 2 \pmod{3}$, then $G$ is word-representable.
\end{theorem}

\begin{proof}
We define
\(
X_i=\{\, j\in V(G)\setminus \{2n-1\}\mid j\equiv i \pmod{3}\,\}, \quad i=0,1,2.
\)
Thus, $X_0,X_1,X_2$ form a partition of $V(G)\setminus\{2n-1\}$. We assign color $c_i$ to the vertices of $X_i$ and color the remaining vertex $2n-1$ with $c_3$.
Let $j\neq 0$ be any vertex in $X_i$. Without loss of generality, let $j\in X_0$. Hence, it has color $c_0$. Now,
\(
j+2n-1 \equiv j-1 \pmod{2n} \quad \text{and} \quad j-1 \equiv 2 \pmod{3}.
\)
Thus, $j$ and $j-1$ receive different colors. Since $3$ does not divide $1, x, 2n-x$, or $n$, similar arguments can be applied to all other adjacencies, as in Theorem~\ref{19}.  
If $j=0$, then $0$ is adjacent to $1,2n-1,x,2n-x$. Here, $1 \in X_{1}$ has color $c_{1}$, the vertices $x$ and $2n-x \in X_{2}$ have color $c_{2}$, and $2n-1$ has color $c_{3}$.  
Hence, none of these neighbours has the same color as vertex $0$. Since $2n-1$ is the only vertex having color $c_3$, none of its adjacent vertices has color $c_3$. 
Hence, $G$ is 4-colorable, with only a single vertex (namely $2n-1$) having color $c_3$.

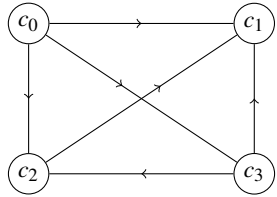
\begin{figure}[h]
\centering
\begin{tikzpicture}[every node/.style={circle, draw, fill=white, inner sep=2pt}]
    \tikzset{
        mid arrow/.style={
            postaction={decorate},
            decoration={markings, mark= at position 0.5 with {\arrow{>}}}
        },
        mid arrow1/.style={
            postaction={decorate},
            decoration={markings, mark= at position 0.4 with {\arrow{>}}}
        },
        mid arrow2/.style={
            postaction={decorate},
            decoration={markings, mark= at position 0.6 with {\arrow{>}}}
        }
    }

    \node (x2) at (0,0) {$c_2$};
    \node (x3) at (3,0) {$c_3$};
    \node (x0) at (0,2) {$c_0$};
    \node (x1) at (3,2) {$c_1$};

    \draw[mid arrow] (x0) -- (x2);
    \draw[mid arrow] (x3) -- (x1);
    \draw[mid arrow] (x3) -- (x2);
    \draw[mid arrow] (x0) -- (x1);

    \draw[mid arrow1] (x0) -- (x3); 
    \draw[mid arrow2] (x2) -- (x1); 
\end{tikzpicture}
\caption{A semi-transitive orientation of \(K_4\)}
\label{k4}
\end{figure}
 We embed the given graph into \(K_4\), the complete graph on four vertices, which is word-representable by Remark \ref{bip}. Let $f: G \rightarrow K_4$ be a homomorphism defined as $f(i)=$ color of the vertex $i$, $\forall i\in \{0,1,\dots,2n-1\}$. We consider an acyclic orientation of \(K_4\) as shown in Figure~\ref{k4}. Clearly, the induced orientation of $G$ through $f$ is acyclic. We claim that this induced orientation is also semi-transitive. Note that the orientation of $K_4$ contains only one shortcutting path: $c_0 \rightarrow c_3 \rightarrow c_2\rightarrow c_1$ with the shortcutting edge: $c_0 \rightarrow c_1$. To prove our claim, it is sufficient to prove that this shortcutting path and edge do not induce any shortcut in $G$. Now, from Table \ref{t1}, it is clear that shortcutting path: $c_0 \rightarrow c_3 \rightarrow c_2 \rightarrow c_1$ induce only one path in G: $0\rightarrow 2n-1\rightarrow 2n-2\rightarrow 2n-3$. But $0\not\sim 2n-3$ as $|2n-3-0| \notin \{1,2n-1,x,2n-x,n\}$. Hence $c_0 \rightarrow c_3 \rightarrow c_2\rightarrow c_1$ with the shortcutting edge: $c_0 \rightarrow c_1$ does not induce any shortcut in $G$. Thus, $G$ is semi-transitive as well as word-representable. 
\begin{table}[h!]
\centering
\begin{tabular}{|c|c|c|c|c|c|c|}
\hline
\textbf{Vertex} & \multicolumn{6}{c|}{\textbf{Adjacent Vertices and Their Colors}} \\ \hline


\multirow{2}{*}{$2n-1$}
& \textbf{Adjacent Vertices} & $0$ & $2n-2$ & $(2n-1+x)_{2n}$ & $(2n-1-x)_{2n}$ & $n-1$ \\ \cline{2-7}
& \textbf{Color} & $c_0$ & $c_2$ & $c_1$ & $c_1$ & $c_1$ \\ \hline

\multirow{2}{*}{$2n-2$}
& \textbf{Adjacent Vertices} & $2n-1$ & $2n-3$ & $(2n-2+x)_{2n}$ & $(2n-2-x)_{2n}$ & $n-2$ \\ \cline{2-7}
& \textbf{Color} & $c_3$ & $c_1$ & $c_0$ & $c_0$ & $c_0$ \\ \hline

\end{tabular}
\vspace{.2cm}
\caption{Adjacencies of the vertex $(2n - 1)$ and $(2n - 2)$ along with their colors}
\label{t1}
\end{table}
\vspace{-1cm}
\end{proof}
\begin{theorem}\label{nx0}
Let $G=C_{2n}(x,1,n)$ be a 5-regular circulant graph. If $n \equiv x \equiv 0 \pmod{3}$ and $x> \tfrac{2n}{3}$, then $G$ is word-representable.
\end{theorem}
\begin{proof}
 We define a partition of \( V(G) \) as follows, 
\(
X = \{0, 1, \dots, x-1\}, \quad 
Y = \{x, x+1, \dots, 2x-1\}, \quad 
Z = \{2x, 2x+1, \dots, 2n-1\}.
\)
For \( i = 0, 1, 2 \), let 
\(
X_i = \{\, j \in X \mid j \equiv i \pmod{3} \,\}.
\)
Analogously, we define \( Y_i \) and \( Z_i \). 
Thus, each of \( X, Y, \) and \( Z \) is further divided into three divisions, namely 
\( X_0, X_1, X_2 \); \( Y_0, Y_1, Y_2 \); and \( Z_0, Z_1, Z_2 \), respectively. 
The coloring corresponding to these divisions is given in Table~\ref{color}. Note that the partition of \(V(G)\) gives \(|X|=|Y|=x\), while 
\(|Z|=2n-2x < x\) because \(x>\tfrac{2n}{3}\).
 Consequently, for any vertex \(j\in V(G)\), the neighbours \((j+x)_{2n}\), \((j+2n-x)_{2n}\),
and \((j+n)_{2n}\) lie in the union of two partitions different from the one containing \(j\).
Therefore, establishing 3-colorability for any one of the partitions \(X\), \(Y\), or \(Z\) is sufficient, 
since the same argument applies symmetrically to all three. Here, we verify it for \(X\).
\begin{figure}[h!]
\centering
\begin{minipage}{0.50\textwidth}
\centering
\renewcommand{\arraystretch}{1.0} 
\begin{tabular}{|c|>{\centering\arraybackslash}m{1cm}|>{\centering\arraybackslash}m{1cm}|>{\centering\arraybackslash}m{1cm}|}
\hline
\textbf{Color} & \multicolumn{3}{c|}{\textbf{Divisions of} \textbf{$X,Y,Z$}} \\ \hline
$c_0$ & $X_0$ & $Y_2$ & $Z_1$ \\ \hline
$c_1$ & $X_1$ & $Y_0$ & $Z_2$ \\ \hline
$c_2$ & $X_2$ & $Y_1$ & $Z_0$ \\ \hline
\end{tabular}
\vspace{.2cm}
\captionof{table}{Color assignments for the divisions of $X, Y, Z$}
\label{color}
\end{minipage}
\hfill
\begin{minipage}{0.45\textwidth}
\centering
\includegraphics[width=0.9\textwidth]{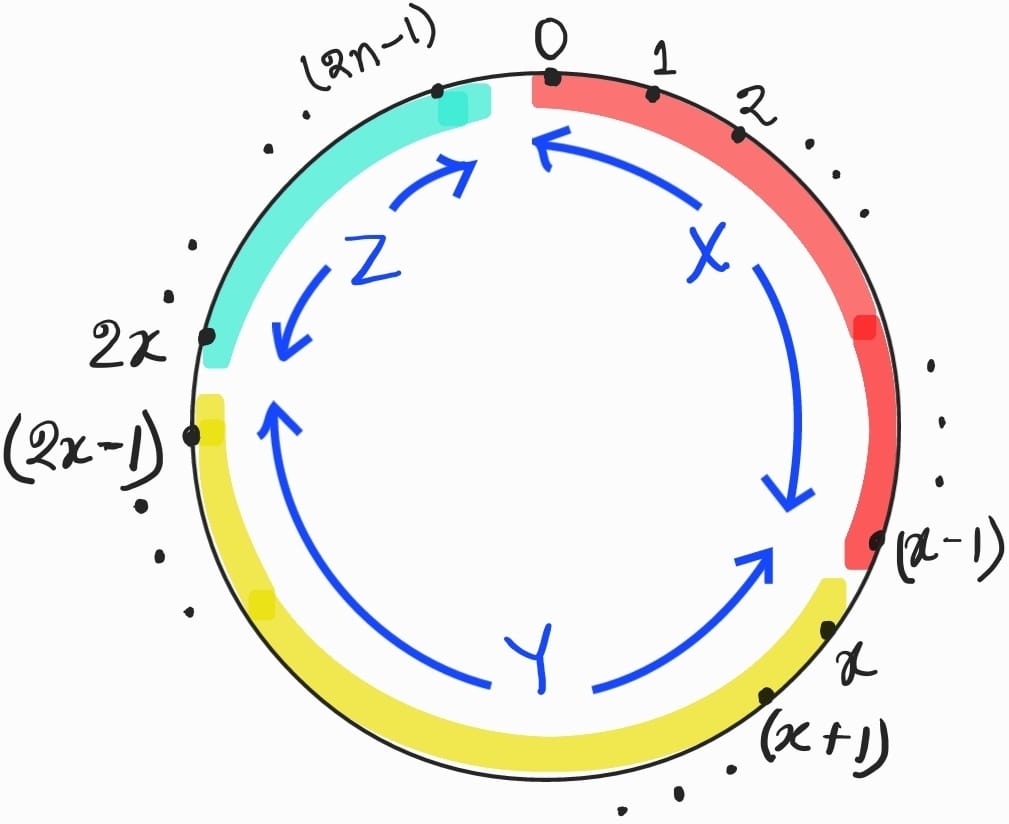}
\captionof{figure}{Schematic structure of the circulant graph with partitions $X,Y$ and $Z$.}
\label{c4}
\end{minipage}
\end{figure}
 \vspace{-.6cm}
 Without loss of generality, let $j\in X_0$. Hence, it has color $c_0$. Since $n\equiv x\equiv0\pmod{3}$ then according to Table \ref{rule}, $ (j+2n-x)_{2n}, (j+x)_{2n}$ and $(j+n)_{2n}$ $\in Y_0\cup Z_0$. Hence, these adjacencies will get colors different from $c_0$. Now, $(j+1)_{2n}\equiv 1\pmod{3}$  implies, $(j+1)_{2n}$ either belongs to $X_1$ or $Y_1$. Hence, it will get either $c_1$ or $c_2$ as its color. Similarly, $(j+2n-1)_{2n}\equiv 2\pmod{3}$ implies, $(j+2n-1)_{2n}$ either belongs to $X_2$ or $Z_2$. Hence, it will get either $c_2$ or $c_1$ as its color. Hence, $G$ is 3-colorable as well as word-representable. 

\end{proof}

\begin{theorem}\label{n0x1}
Let $G=C_{2n}(x,1,n)$ be a 5-regular circulant graph. If $n \equiv 0\pmod{3},  x \equiv 1 \pmod{3}$ and $x< \tfrac{n}{2}$, then $G$ is word-representable.
\end{theorem}
\begin{proof}
 We define a partition of \( V(G) \) as follows, 
\(
X = \{0, 1, \dots, n-1\}, \quad 
Y = \{n, n+1, \dots, 2n-x-1\}, \quad 
Z = \{2n-x, 2n-x+1, \dots, 2n-1\}.
\)  
For \( i = 0, 1, 2 \), let 
\(
X_i = \{\, j \in X \mid j \equiv i \pmod{3} \,\}.
\)
Analogously, we define \( Y_i \) and \( Z_i \).
Thus, each of \( X, Y, \) and \( Z \) is further divided into divisions, namely 
\( X_0, X_1, X_2 \); \( Y_0, Y_1, Y_2 \); and \( Z_0, Z_1, Z_2 \), respectively. The coloring corresponding to these divisions is given in Table~\ref{color}.  For additional clarity, we refer to Figure~\ref{c3} for the partitions \(X, Y, Z\). 
(Note that the coloring of the vertices \(n\), \(2n-x\), and \(2n-1\) with color \(c_{3}\) in  Figure \ref{c3} is not applicable in the present context.)

  First we list down some observations which follow from the condition $x<\tfrac{n}{2}$ and the partitions of $V(G)$.
 \begin{enumerate}
     \item For any vertex $j\in X$, $(j+2n-x)_{2n}$  $\in X\cup Z$. This is because, $0\leq j\leq n-1  \implies 2n-x \leq (j+2n-x)_{2n}\leq (3n-x-1)_{2n} \implies 2n-x \leq (j+2n-x)_{2n}\leq n-x-1 < n-1$.
     
     From the interval definitions of \( X, Y, Z \), we obtain the following inclusions:
     
- if \( j \in X \), then \( (j+n)_{2n} \in Y \cup Z \);

- if \( j \in Y \cup Z\), then \( (j+n)_{2n} \in X \);

- if \( j \in Z \), then \( (j+x)_{2n} \in X \);

- if \( j \in Y \), then \( (j+2n-x)_{2n} \in X \cup Y \);

- if \( j \in Y \), then \( (j+x)_{2n} \in Y \cup Z \).

     \item For any vertex $j\in X$, $(j+x)_{2n}$ $\in X\cup Y$. Indeed, since \( 0 \le j \le n-1 \), we have $x \leq (j+x)_{2n}\leq n-1+x< 2n-x-1$ as $x<\tfrac{n}{2}$. Similarly, For any vertex $j\in Z$, $(j+2n-x)_{2n}$ $\in Y$.   
 \end{enumerate}
 Using the above observations together with Table~\ref{color}, we see that for any vertex $ j\in V(G)$ the neighbours 
\( (j+x)_{2n} \), \( (j+2n-x)_{2n} \), and \( (j+n)_{2n} \) always receive colors distinct from that of \( j \).  
It remains to verify that the same holds for the vertices \( j+1 \) and \( j+2n-1 \).  
To check this, we assume without loss of generality that \( j \in X_0 \). Hence, it has color $c_0$ Now, $(j+1)_{2n}\equiv 1\pmod{3}$  implies, $(j+1)_{2n}$ either belongs to $X_1$ or $Y_1$. Hence, it will get either $c_1$ or $c_2$ as its color. Similarly, $(j+2n-1)_{2n}\equiv 2\pmod{3}$ implies, $(j+2n-1)_{2n}$ either belongs to $X_2$ or $Z_2$. Hence, it will get either $c_2$ or $c_1$ as its color. Similar arguments will hold for $Y$ and $Z$. Hence, $G$ is 3-colorable as well as word-representable.  
\end{proof}
\begin{theorem}\label{n0x2}
Let $G=C_{2n}(x,1,n)$ be a 5-regular circulant graph. If $n \equiv 0\pmod{3},  x \equiv 2 \pmod{3}$ and $2<x< \tfrac{n}{2}$, then $G$ is word-representable.
\end{theorem}
\begin{proof}
We use the same partition as in Theorem~\ref{n0x1} and follow the coloring scheme in Table~\ref{color1}, 
except that the vertices $n$, $2n-x$, and $2n-1$ are assigned the additional color $c_3$, 
as shown in Figure~\ref{c3}. These three vertices are not adjacent to one another since 
their pairwise differences  $\notin\{1,2n-1,x,2n-x,n\}$.

With this small change, the same observations used in Theorem~\ref{n0x1}, combined with the coloring scheme in Table~\ref{color1}, imply that for every vertex $j\in V(G)$, the neighbours 
$(j+x)_{2n}$, $(j+2n-x)_{2n}$, and $(j+n)_{2n}$ always receive colors different from that of $j$.

Furthermore, for the adjacencies \( (j+1)_{2n} \) and \( (j+2n-1)_{2n} \), the same argument used in Theorem~\ref{n0x1} holds. The only exceptions occur when these neighbours coincide with one 
of the specially colored vertices. Specifically,

-if $ j \in \{2n-2,\; n-1,\; 2n-x-1\} $\;then $(j+1) _{2n}$ has color $ c_3$.

-if $ j \in \{0,\; n+1,\; 2n-x+1\} $\;then $(j+2n-1) _{2n}$ has color $ c_3$.

 Hence, $G$ is 4-colorable.
 \begin{figure}[h!]
\centering
\begin{minipage}{0.36\textwidth}
\centering
\renewcommand{\arraystretch}{1.0}
\begin{tabular}{|c|>{\centering\arraybackslash}m{1cm}|>{\centering\arraybackslash}m{1cm}|>{\centering\arraybackslash}m{1cm}|}
\hline
\textbf{Color} & \multicolumn{3}{c|}{\textbf{Divisions of} \textbf{$X,Y,Z$}} \\ \hline
$c_0$ & $X_0$ & $Y_1$ & $Z_2$ \\ \hline
$c_1$ & $X_1$ & $Y_2$ & $Z_0$ \\ \hline
$c_2$ & $X_2$ & $Y_0$ & $Z_1$ \\ \hline
\end{tabular}
\captionof{table}{Color assignments for the divisions of $X, Y, Z$ except the vertices $n, 2n-x$ and
$2n-1$.}
\label{color1}
\end{minipage}
\hfill
\begin{minipage}{0.45\textwidth}
\centering
\includegraphics[width=0.9\textwidth]{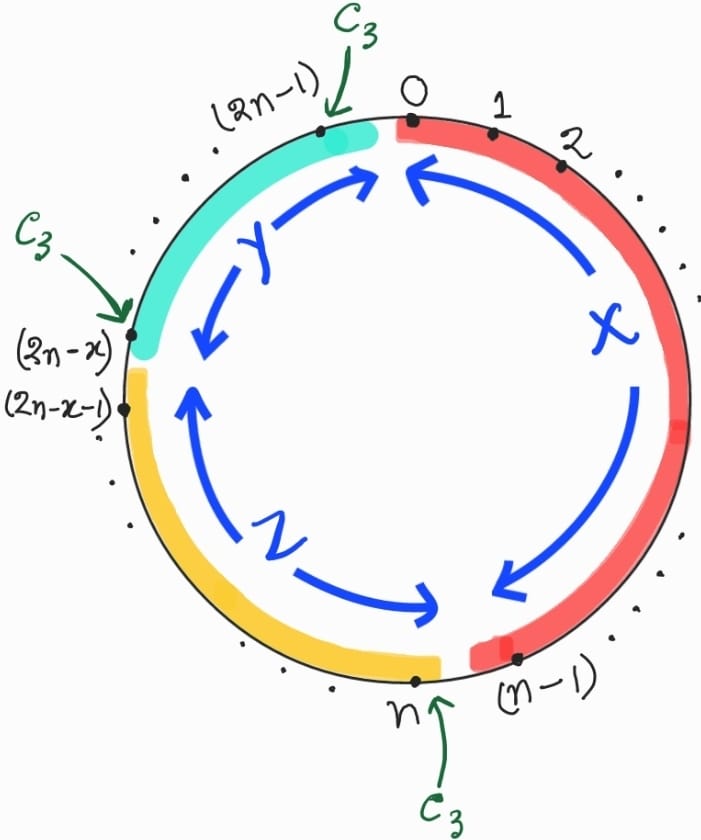}
\captionof{figure}{Schematic structure of the circulant graph with partitions $X,Y,Z$ and vertices $n, 2n-x, 2n-1$ with color $c_3$.}
\label{c3}
\end{minipage}
\end{figure}
 Now, we embed this graph into $K_4$ as we did in Theorem \ref{nx2} using Figure \ref{k4}. We next verify that, the only shortcutting path: $c_0 \rightarrow c_3 \rightarrow c_2\rightarrow c_1$ with the shortcutting edge: $c_0 \rightarrow c_1$, does not induce any shortcut in $G$. The verification is carried out in three cases.

 \medskip
\noindent\textbf{Case 1.}
Assume that the vertex with color \(c_3\) is \(n\).
The corresponding shortcutting paths in \(G\) that follow the color sequence are
$
0 \rightarrow n \rightarrow n-1 \rightarrow n-2
\quad\text{and}\quad
n+1 \rightarrow n \rightarrow n-1 \rightarrow n-2.
$
But, there is no edge between \(0\) and \(n-2\), nor between \(n+1\) and \(n-2\) as their absolute differences  $\notin\{1,2n-1,x,2n-x,n\}$. 

\noindent\textbf{Case 2.}
Assume that the vertex with color \(c_3\) is \(2n-x\).
The corresponding shortcutting paths in \(G\) that follow the color sequence are
$
0 \rightarrow 2n-x \rightarrow 2n-x-1 \rightarrow 2n-x-2
\quad\text{and}\quad
2n-x+1 \rightarrow 2n-x \rightarrow 2n-x-1 \rightarrow 2n-x-2.
$
But, there is no edge between \(0\) and \(2n-x-2\), nor between \(2n-x+1\) and \(2n-x-2\) as their absolute differences  $\notin\{1,2n-1,x,2n-x,n\}$.

\noindent\textbf{Case 3.}
Assume that the vertex with color \(c_3\) is \(2n-1\).
The corresponding shortcutting paths in \(G\) that follow the color sequence are
$
0 \rightarrow 2n-1 \rightarrow 2n-x-1 \rightarrow 2n-x-2, 0 \rightarrow 2n-1 \rightarrow n-1 \rightarrow n-2, 0 \rightarrow 2n-1 \rightarrow 2n-2 \rightarrow 2n-3, 0 \rightarrow 2n-1 \rightarrow 2n-2 \rightarrow 2n-2-x
\quad\text{and}\quad
0 \rightarrow 2n-1 \rightarrow 2n-2 \rightarrow n-2.
$
But, there is no edge between \(0\) and \(2n-x-2\), \(0\) and \(n-2\), \(0\) and \(2n-3\) as their absolute differences  $\notin\{1,2n-1,x,2n-x,n\}$.

\noindent Hence, $G$ is word-representable.
\end{proof}
\begin{theorem}\label{n1x2}
Let $G=C_{2n}(x,1,n)$ be a 5-regular circulant graph. If $n \equiv 1\pmod{3},  x \equiv 2 \pmod{3}$ and $2<x< \tfrac{n}{2}$, then $G$ is word-representable.
\end{theorem}
\begin{proof}
We define a partition of \( V(G) \) as follows, 
\(
X = \{0, 1, \dots, 2n-x-1\}, \quad 
Y = \{2n-x, 2n-x+1, \dots, 2n-1\}.
\)
For \( i = 0, 1, 2 \), let 
\(
X_i = \{\, j \in X \mid j \equiv i \pmod{3} \,\}.
\)
Analogously, we define \( Y_i \).
Thus, each of \( X \) and \( Y \) is further divided into divisions, namely 
\( X_0, X_1, X_2 \) and \( Y_0, Y_1, Y_2 \), respectively.  We assign color $c_i$ to the vertices of $X_i$ and  $d_i$ to the vertices of $Y_i.$ Since, $|Y|=x$ for $j\in Y$, $ (j+2n-x)_{2n}, (j+x)_{2n}$ and $(j+n)_{2n}$ $\in X$. Hence, they will get colors  different from that of $j$. Without loss of generality, let $j\in Y_0$. Hence, it has color $d_0$. Now, $(j+1)_{2n}\equiv 1\pmod{3}$  implies, $(j+1)_{2n}$ either belongs to $Y_1$ or $X$. Hence, it will get a color different from $d_0$. Similarly, $(j+2n-1)_{2n}\equiv 2\pmod{3}$ implies, $(j+2n-1)_{2n}$ either belongs to $Y_2$ or $X$. Hence, it receives a color distinct from $d_0$. The same argument applies to any vertex $j \in X$ with respect to its neighbors $(j+1)_{2n}$ and $(j+2n-1)_{2n}$, when $j\in X_0$. Hence, it suffices to prove that the same holds for any $j\in X$ with respect to its neighbours $(j+x)_{2n},(j+2n-x)_{2n}$ and $(j+n)_{2n}$. Without loss of generality, let $j\in X_0$. Hence, it has color $c_0$.It follows from the definition of $X,Y$ and Table \ref{rule}, that for any $j\in X_{0},$ its neighbours $(j+1)_{2n}$ and $(j+2n-1)_{2n}$ $\in X_{1}\cup X_{2}\cup Y$. Hence, Their colors are different from $c_0$. Hence, $G$ is $6$-colorable.

\begin{figure}[h!]
\centering
\begin{minipage}{0.43\textwidth}
\centering
\includegraphics[width=0.9\textwidth]{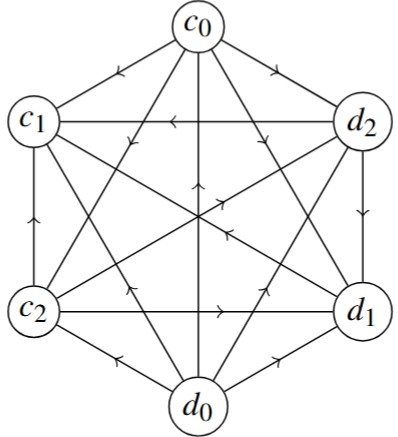}
\captionof{figure}{A semi-transitive orientation of $K_6$.}
\label{k6}
\end{minipage}
\hfill
\begin{minipage}{0.49\textwidth}
\centering
\includegraphics[width=0.9\textwidth]{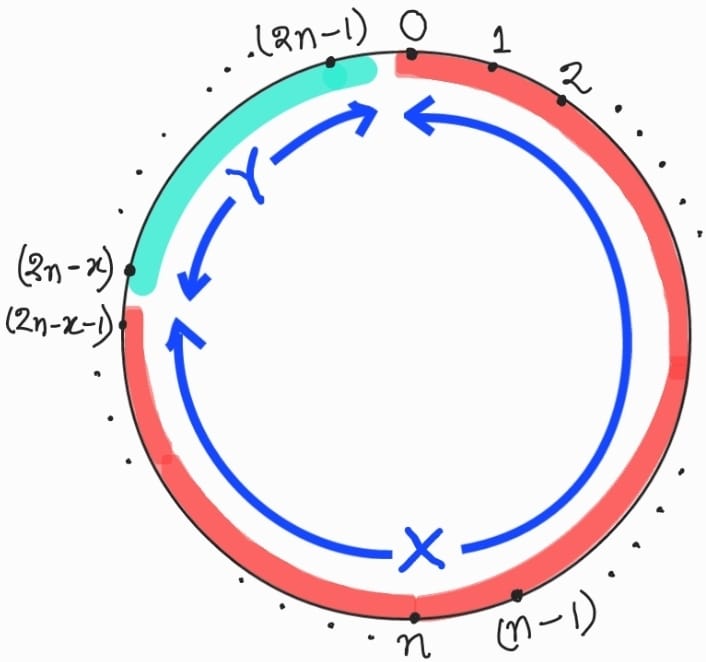}
\captionof{figure}{Schematic structure of the circulant graph with partitions $X$ and $Y$.}
\end{minipage}
\end{figure}

We now embed this graph $G$ into $K_6$, complete graph on 6 vertices using homomorphism $f: G \rightarrow K_6$ such that $f(i)=$ color of the vertex $i$, $\forall i\in \{0,1,\dots,2n-1\}$. We consider an acyclic orientation of \(K_6\) as shown in Figure~\ref{k6}. Clearly, the induced orientation of $G$ through $f$ is acyclic. We claim that this induced orientation is also semi-transitive.
Note that, topological sorting of the orientation in Figure \ref{k6} is $d_0 \rightarrow c_0 \rightarrow d_2\rightarrow c_2\rightarrow d_1\rightarrow c_1$. Since length of any shortcut is $\geq 4$, the number of shortcutting paths $=\binom{6}{4} + \binom{6}{5} + \binom{6}{6}
=22.$ Next, we state a claim that allows us to eliminate \(13\) possible cases of shortcutting paths in \(G\).
\begin{claim}
   \textbf{1.} For any path in \(G\) which is induced by 
\begin{itemize}
    \item  \(d_{2} \rightarrow c_{2} \rightarrow  d_{1}\) , the \(c_{2}\)-colored vertex in this path has no \(d_{1}\)-colored neighbor.
   \item \(d_{0} \rightarrow c_{0} \rightarrow  d_{2}\) , the \(c_{0}\)-colored vertex in this path has no \(d_{2}\)-colored neighbor.
    \item  \(d_{0} \rightarrow c_{2} \rightarrow  d_{1}\) , the \(c_{2}\)-colored vertex in this path has no \(d_{1}\)-colored neighbor.
\end{itemize}
\textbf{2.} For any path in \(G\) which is induced by  \(d_{0} \rightarrow c_{0} \rightarrow  c_{2}\rightarrow  d_{1}\) , the \(c_{2}\)-colored vertex in this path has no \(d_{1}\)-colored neighbor.\\
\smallskip
\textbf{3.} The path \(d_{0} \rightarrow c_{0} \rightarrow  c_{2}\rightarrow  d_{1}\) does not induce any shortcut in $G$.\\
\end{claim}
\vspace{-.5cm}
\textit{Proof of the Claim:}
\textbf{1.} The proofs for the three subcases of the claim follow by similar reasoning. Therefore, without loss of generality, we provide the proof for the first subcase. Let \(j \in Y\) be any \(d_{2}\)-colored vertex. By Table~\ref{rule}, the unique 
\(c_{2}\)-colored neighbor of a \(d_{2}\)-colored vertex is \((j+x)_{2n}\). Since 
\(|Y| = x\), we have \((j+x)_{2n} \in X\), and moreover \((j+x)_{2n} < x\). 

Since \((j+x)_{2n} < x\) and \(|X| = 2n - x > 2n - \tfrac{n}{2} = \tfrac{3n}{2}\), all 
neighbors of \((j+x)_{2n}\)—namely \((j+2x)_{2n}\), \((j+x+2n-1)_{2n}\), \((j+x+1)_{2n}\), and \((j+n)_{2n}\)—also lie 
in \(X\). However, every \(d_{1}\)-colored vertex belongs to \(Y\). Therefore, 
\((j+x)_{2n}\) has no neighbor with color \(d_{1}\).

Hence no further extension of these paths consistent with the topological sorting  
\(d_{0} \rightarrow c_{0} \rightarrow d_{2} \rightarrow c_{2} \rightarrow d_{1} \rightarrow c_{1}\)  
can induce a shortcut in \(G\). Therefore using this, we can  discard a total of 
\(7 + 5 + 1 = 13\) potential shortcutting paths.

\textbf{2.} According to Table \ref{rule}, there are two possible induced paths in $G$ given below.
\vspace{-.1cm}
\begin{itemize}
    \item \(j \rightarrow (j+x)_{2n} \rightarrow  (j+x-1)_{2n}\),
    \item \(j \rightarrow (j+x)_{2n} \rightarrow  (j+2x)_{2n}\).
    
\end{itemize}
Now, the first case can be proved exactly as before. For the second case, since 
\(|X|>\tfrac{3n}{2}\), all neighbors of \((j+2x)_{2n}\) except \((j+2x+n)_{2n}\) lie in \(X\). 
Moreover, \((j+2x+n)_{2n} \in X_{0} \cup Y_{0}\). Therefore, 
\((j+2x)_{2n}\) has no neighbor with color \(d_{1}\).
Using this, we discard two additional potential shortcutting paths—namely \(d_{0} \rightarrow c_{0} \rightarrow  c_{2}\rightarrow  d_{1}\) and \(d_{0} \rightarrow c_{0} \rightarrow  c_{2}\rightarrow  d_{1}\rightarrow  c_{1}\).

\textbf{3.} According to Table \ref{rule}, there are two possible cases.
\vspace{-.1cm}
\begin{itemize}
   \item If \( j \neq (2n - x) \), then the \(c_{0}\)-colored vertex in this path has no \(d_{1}\)-colored neighbor, and the argument proceeds exactly as in \textbf{1}.

\item If \( j = (2n - x) \), then the induced path in \(G\) is 
\(
(2n - x) \;\rightarrow\; 0 \;\rightarrow\; (2n - 1) \;\rightarrow\; (x - 1).
\)
However, there is no edge between \(2n - x\) and \(x - 1\), since 
\(
|(2n - x) - (x - 1)| = |2n - 2x + 1|
\)
 \(\in \{1,\, 2n-1,\, x,\, 2n-x,\, n\}\). Thus, no shortcut arises in this case.

\end{itemize}

Up to this point, we have eliminated a total of \(13 + 2 + 1 = 16\) potential shortcutting paths.  
The remaining possible $6$ cases are 
(i) $d_{0} \rightarrow c_{0} \rightarrow  c_{2}\rightarrow  c_{1}$, (ii) $d_{0} \rightarrow d_{2} \rightarrow  c_{2}\rightarrow  c_{1}$, (iii) $d_{0} \rightarrow d_{2} \rightarrow  d_{1}\rightarrow  c_{1}$, (iv) $c_{0} \rightarrow c_{2} \rightarrow  d_{1}\rightarrow  c_{1}$, (v) $c_{0} \rightarrow d_{2} \rightarrow  c_{2}\rightarrow  c_{1}$ and (vi) $c_{0} \rightarrow d_{2} \rightarrow  d_{1}\rightarrow  c_{1}$.Since the arguments for all remaining cases are analogous, we prove (i) without loss of generality.

In this case, the induced path in \(G\) is either 
\(
j \rightarrow (j+x)_{2n} \rightarrow (j+x-1)_{2n} \rightarrow (j+2x-1)_{2n}
\)
or 
\(
j \rightarrow (j+x)_{2n} \rightarrow (j+2x)_{2n} \rightarrow (j+3x)_{2n}.
\)
In both paths, the absolute difference between the first and the last vertices   
\(\notin \{1,\, 2n-1,\, x,\, 2n-x,\, n\}\). Consequently, all \(22\) potential shortcutting paths fail to create a shortcut, and hence \(G\) is word-representable.
\end{proof}
\begin{theorem}\label{n2x1}
Let $G=C_{2n}(x,1,n)$ be a 5-regular circulant graph. If $n \equiv 2\pmod{3},  x \equiv 1 \pmod{3}$ and $x< \tfrac{n}{2}$, then $G$ is word-representable.
\end{theorem}
\begin{proof}
    The proof proceeds exactly as in Theorem~\ref{n1x2}, provided we use the orientation of \(K_6\) that follows the topological sorting
\(
d_{0} \;\rightarrow\; c_{0} \;\rightarrow\; d_{1} \;\rightarrow\; c_{1} \;\rightarrow\; d_{2} \;\rightarrow\; c_{2}.
\)

\end{proof}
\subsection{Word-Representability and the Representation Number of 5-Regular Circulant Graph Using Morphisms}\label{mor}
This section addresses the word-representability and the representation number of certain subclasses of \(5\)-regular circulant graphs. The proofs primarily follow the approach via morphisms as presented in \cite{srinivasan2025semi}. For any two words \(w\) and \(u\), we write
\(
\underbrace{w}_{u}
\)
to indicate that \(u\) is a subword of \(w\).  
Moreover, for two letters \(y, z\) in a word \(w\), we write \(y \prec_w z\) if \(y\) occurs to the left of \(z\) in \(w\).

\begin{theorem}\label{20}
    Let $G=C_{2n}(x,1,n)$ be a 5-regular circulant graph. Then $G$ is word-representable for $\frac{n}{2} < x \leq \frac{2n}{3}$ and $R(G) \leq 5.$
    \begin{proof}
        
Consider the connected $5$-regular circulant graph $G \cong C(x,1,n)$ with 
\(
\frac{n}{2} < x \leq \frac{2n}{3}.
\)  Since $G$ is a 5-regular circulant graph, then $1<x<n$ and hence $n\geq3.$
The vertex set of $G$ is 
\(
V(G) = \{0,1,2,\ldots,n-1\}.
\) 
By the definition of a circulant graph, each vertex $i \in V(G)$ is adjacent to each of
\(
(i+1)_{2n}, (i-1)_{2n}, (i+x)_{2n},  (i-x)_{2n}\) and \((i+n)_{2n}.
\) Define a morphism $f : V(G)^{\ast} \to V(G)^{\ast}$ by
\[
f(i) = (i)_{2n}(i-1)_{2n}(i-x)_{2n}(i+n)_{2n}(i+x)_{2n},
\forall i \in V(G).
\]
We show that the word \(f(w)\) represents \(G\), where
\(
w = 0\,1\,2\,\ldots\,(n-1).
\)
For each \(0 \leq i \leq n-1\), let
\(
w_i = i\,(i+1)\,\ldots\,(n-1)\,0\,1\,\ldots\,(i-1).
\)
By Proposition~\ref{pr1}, \(f(w_i)\) represents \(G\) for all \(0 \leq i \leq n-1\). Consequently, two vertices \(i\) and \(j\) alternate in \(f(w)\) if and only if they alternate in \(f(w_i)\) for every \(i\).
From the definition of the morphism \(f\), each vertex \(i\) occurs exactly once in each of the words \(f(i)\), \(f((i+1)_{2n})\), \(f((i+x)_{2n})\), \(f((i+n)_{2n})\), and \(f((i-x)_{2n})\). Hence, \(f(w)\) is a 5-uniform word. Moreover,
\[
f(w_i) = (i)\,(i-1)_{2n}\,(i-x)_{2n}\,(i+n)_{2n}\,(i+x)_{2n}\,(i+1)_{2n}\,(i)\,\ldots\,(i-1)_{2n}\,(i-2)_{2n}\,(i-1-x)_{2n}\]\[\hspace{-6,1cm}
\,(i-1+n)_{2n}\,(i-1+x)_{2n}.
\]

For any vertex \(i \in V(G)\), the occurrence of the factor \(f(i)\,f(i+1)\) in \(f(w_i)\) guarantees that \(i\) does not alternate with any vertex
\(
j \in V(G) \setminus \{(i+1)_{2n},\; (i-1)_{2n},\; (i+x)_{2n},\; (i-x)_{2n},\; (i+n)_{2n}\}.
\)
Therefore, it is sufficient to verify that \(i\) alternates with each vertex
\(
j \in \{(i+1)_{2n},\; (i-1)_{2n},\; (i+x)_{2n},\; (i-x)_{2n},\; (i+n)_{2n}\}
\)
within \(f(w_i)\), for all \(0 \leq i \leq n-1\).

\textbf{Case(1):} Here we will show the alternation between $i$ and $(i+1)_{2n}$ in $f(w_i)$.
Since $\tfrac{n}{2} < x \leq \tfrac{2n}{3}$ and $1<x<n$, we have
\(
(i+x+1)_{2n} \;\preceq_{w_i}\; (i+n)_{2n}\) and
\( 
(i+n+1)_{2n} \;\preceq_{w_i}\; (i-x)_{2n}.
\)

If possible let, $(i+n)_{2n} \;\prec_{w_i}\; (i+x+1)_{2n}$ and $(i-x)_{2n} \;\prec_{w_i}\; (i+n+1)_{2n}$, i.e, $i+n<i+x+1$ and $2n+i-x<i+n+1$. Both of these lead to $n<x+1$, which is a contradiction as $x<n$.  
Hence, for every $0 \leq i \leq 2n-1$, the following ordering holds in $w_i$:  
\[
i \prec_{w_i}(i+1)_{2n}\prec_{w_i}(i+2)_{2n}\prec_{w_i}(i+x)_{2n}\prec_{w_i}(i+x+1)_{2n}\preceq_{w_i}(i+n)_{2n}\prec_{w_i}(i+n+1)_{2n}\preceq_{w_i}\]\[\hspace{-8.5cm}(i-x)_{2n}\prec_{w_i}(i-x+1)_{2n}.\]

By the definition of the morphism $f$, we have four subcases,
\begin{itemize}
    \item $(i+x+1)_{2n}=(i+n)_{2n}$ and $(i+n+1)_{2n}=(i-x)_{2n}$
    \[
f(w_i) 
= \underbrace{f(i)}_{i}\;\;
   \underbrace{f(i+1)}_{(i+1)_{2n},i}\;\;
   \underbrace{f(i+2)}_{(i+1)_{2n}}\cdots
   \underbrace{f(i+x)}_{i}\;\;
   \underbrace{f(i+x+1)}_{(i+1)_{2n},i}\;\;
   \underbrace{f(i+n+1)}_{(i+1)_{2n},i}\cdots
   \underbrace{f(i-x+1)}_{(i+1)_{2n}}.
\]
    \item $(i+x+1)_{2n}\neq(i+n)_{2n}$ and $(i+n+1)_{2n}=(i-x)_{2n}$
    \[
f(w_i) 
= \underbrace{f(i)}_{i}\;\;
   \underbrace{f(i+1)}_{(i+1)_{2n},i}\;\;
   \underbrace{f(i+2)}_{(i+1)_{2n}}\cdots
   \underbrace{f(i+x)}_{i}\;\;
   \underbrace{f(i+x+1)}_{(i+1)_{2n}}\cdots
   \underbrace{f(i+n)}_{i}
   \underbrace{f(i+n+1)}_{(i+1)_{2n},i}\cdots\]
   \[\hspace{-7cm}\underbrace{f(i-x+1)}_{(i+1)_{2n}}.
\]

    \item $(i+x+1)_{2n}=(i+n)_{2n}$ and $(i+n+1)_{2n}\neq(i-x)_{2n}$
    \[
f(w_i) 
= \underbrace{f(i)}_{i}\;\;
   \underbrace{f(i+1)}_{(i+1)_{2n},i}\;\;
   \underbrace{f(i+2)}_{(i+1)_{2n}}\cdots
   \underbrace{f(i+x)}_{i}\;\;
   \underbrace{f(i+x+1)}_{(i+1)_{2n},i}\;\;
   \underbrace{f(i+n+1)}_{(i+1)_{2n}}\cdots\cdots
   \underbrace{f(i-x)}_{i}\]\[\hspace{-7cm}
   \underbrace{f(i-x+1)}_{(i+1)_{2n}}.
\]
    \item $(i+x+1)_{2n}\neq(i+n)_{2n}$ and $(i+n+1)_{2n}\neq(i-x)_{2n}$
     \[
f(w_i) 
= \underbrace{f(i)}_{i}\;\;
   \underbrace{f(i+1)}_{(i+1)_{2n},i}\;\;
   \underbrace{f(i+2)}_{(i+1)_{2n}}\cdots
   \underbrace{f(i+x)}_{i}\;\;
   \underbrace{f(i+x+1)}_{(i+1)_{2n}}\cdots
   \underbrace{f(i+n)}_{i}\;\;
   \underbrace{f(i+n+1)}_{(i+1)_{2n}}\cdots\]
   \[\hspace{-6cm}
   \underbrace{f(i-x)}_{i}
   \underbrace{f(i-x+1)}_{(i+1)_{2n}}.\]
\end{itemize}

Hence, $i$ and $(i+1)_{2n}$ alternates in $f(w_i)$ for all $0\leq i\leq 2n-1$

The alternation between $i$ and $(i-1)_{2n}$, $i$ and $(i+x)_{2n}$, $i$ and $(i-x)_{2n}$, as well as $i$ and $(i+n)_{2n}$ can be proved in a similar manner. Hence, in all cases, the required alterations hold. Thus, the word $f(w)$ represents the graph $G$ and consequently $R(G)\leq 5.$ For a detailed proof, we refer the reader to the appendix.

\end{proof}
\end{theorem}
 \begin{theorem}\label{21}
    Let $G=C_{2n}(x,1,n)$ be a 5-regular circulant graph and $x=\frac{n}{2}$. Then $G$ is word-representable and $R(G) \leq 5.$
    \begin{proof}     
Consider the connected $5$-regular circulant graph $G \cong C(x,1,n)$ with 
$x=\frac{n}{2}$. Since $G$ is a 5-regular circulant graph, then $1<x<n$, i.e, $x\geq2$ and $n\geq4$.
The vertex set of $G$ is 
\(
V(G) = \{0,1,2,\ldots,n-1\}.
\) 
By the definition of a circulant graph, each vertex $i \in V(G)$ is adjacent to each of
\(
(i+1)_{2n}, (i-1)_{2n}, (i+x)_{2n},  (i-x)_{2n}\) and \((i+n)_{2n}.
\) Define a morphism $f : V(G)^{\ast} \to V(G)^{\ast}$ by
\[
f(i) = (i)_{2n}(i-1)_{2n}(i+x)_{2n}(i+n)_{2n}(i-x)_{2n},
\forall i \in V(G).
\]
We shall show that the word \(f(w)\) represents \(G\), where
\(
w = 0\,1\,2\,\ldots\,(n-1).
\)
For each \(0 \leq i \leq n-1\), define
\(
w_i = i\,(i+1)\,\ldots\,(n-1)\,0\,1\,\ldots\,(i-1).
\)
By Proposition~2.6, the word \(f(w_i)\) represents \(G\) for all \(0 \leq i \leq n-1\). Consequently, two vertices \(i\) and \(j\) alternate in \(f(w)\) if and only if they alternate in \(f(w_i)\) for every \(i\).
From the definition of the morphism \(f\), each vertex \(i\) occurs exactly once in each of the words
\(
f(i),\; f((i+1)_{2n}),\; f((i+x)_{2n}),\; f((i+n)_{2n}),\; \text{and } f((i-x)_{2n}).
\)
Hence, \(f(w)\) is a 5-uniform word. Moreover,
\[
f(w_i) = (i)\,(i-1)_{2n}\,(i+x)_{2n}\,(i+n)_{2n}\,(i-x)_{2n}\,(i+1)_{2n}\,(i)\,\ldots\,(i-1)_{2n}\,(i-2)_{2n}\,(i-1+x)_{2n}\]\[\hspace{-6.2cm}\,(i-1+n)_{2n}\,(i-1-x)_{2n}.
\]

For any vertex \(i \in V(G)\), the presence of the factor \(f(i)\,f(i+1)\) in \(f(w_i)\) guarantees that \(i\) does not alternate with any vertex
\(
j \in V(G) \setminus \{(i+1)_{2n},\; (i-1)_{2n},\; (i+x)_{2n},\; (i-x)_{2n},\; (i+n)_{2n},\;(i+n+1)_{2n} \}.
\)
Since $i \nsim (i+n+1)_{2n}$, we now show that $i$ does not alternate with $(i+n+1)_{2n}$ in $f(w_i)$. Observe that, for every $0 \le i \le 2n-1$, the following relative ordering appears in $w_i$:
\[
i\prec_{w_i}(i+1)_{2n}\prec_{w_i}(i+2)_{2n}\prec_{w_i}(i+n)_{2n}\prec_{w_i}(i+n+1)_{2n}\prec_{w_i}(i+n+3)_{2n} \]
By the definition of the morphism $f$, we have
\[
f(w_i) 
= \underbrace{f(i)}_{i}\;\;
   \underbrace{f(i+1)}_{(i+n+1)_{2n},i}\;\;
   \underbrace{f(i+2)}_{i}\cdots
   \underbrace{f(i+n)}_{i}\;\;
   \underbrace{f(i+n+1)}_{(i+n+1)_{2n}}\cdots
   \underbrace{f(i+n+3}_{(i+n+1)_{2n}}
\]
Hence, $i$ does not alternate with $(i+n+1)_{2n}$ in $f(w_i)$.
Therefore, it suffices to verify that \(i\) alternates with each vertex
\(
j \in \{(i+1)_{2n},\; (i-1)_{2n},\; (i+x)_{2n},\; (i-x)_{2n},\; (i+n)_{2n}\}
\)
within \(f(w_i)\), for all \(0 \leq i \leq n-1\).

\textbf{Case(1):} Here we will show the alternation between $i$ and $(i+1)_{2n}$ in $f(w_i)$.
Since $x=\frac{n}{2}$ and $x\geq2$ and $n\geq4$, we have
\(
(i+x+1)_{2n} \;\prec_ {w_i}\; (i+n)_{2n}\) and \((i+n+1)_{2n} \;\prec_ {w_i}\; (i-x)_{2n}.
\)
If possible let, $(i+n)_{2n} \;\preceq_{w_i}\; (i+x+1)_{2n}$ and $(i-x)_{2n} \;\preceq_{w_i}\; (i+n+1)_{2n}$, i.e, $i+n\leq i+x+1$ and $2n+i-x\leq i+n+1$. Both of these leads to $n\leq x+1$, i.e, $n\leq 2$, which is a contradiction.
Hence, for every $0 \leq i \leq 2n-1$, the following ordering holds in $w_i$:  
\[
i\prec_{w_i}(i+1)_{2n}\prec_{w_i}(i+2)_{2n}\preceq_{w_i}(i+x)_{2n}\prec_{w_i}(i+x+1)_{2n}\prec_{w_i}(i+n)_{2n}\prec_{w_i}(i+n+1)_{2n}\]\[\hspace{-8cm}\prec_{w_i}(i-x)_{2n}\prec_{w_i}(i-x+1)_{2n} \]
By the definition of the morphism $f$, we have two subcases,
\begin{itemize}
    \item $(i+2)_{2n}=(i+x)_{2n}$
    \[
f(w_i) 
= \underbrace{f(i)}_{i}\;\;
   \underbrace{f(i+1)}_{(i+1)_{2n},i}\;\;
   \underbrace{f(i+2)}_{(i+1)_{2n},i}\;\;
   \underbrace{f(i+x+1)}_{(i+1)_{2n}}\cdots
   \underbrace{f(i+n)}_{i}\;\;
   \underbrace{f(i+n+1)}_{(i+1)_{2n}}\cdots
   \underbrace{f(i-x)}_{i}\;\;\]\[\hspace{-7.3cm}
   \underbrace{f(i-x+1)}_{(i+1)_{2n}}
\]
    \item $(i+2)_{2n}\neq (i+x)_{2n}$    \[
f(w_i) 
= \underbrace{f(i)}_{i}\;\;
   \underbrace{f(i+1)}_{(i+1)_{2n},i}\;\;
   \underbrace{f(i+2)}_{(i+1)_{2n}}\cdots
  \underbrace{f(i+x)}_{i}\;\;
  \underbrace{f(i+x+1)}_{(i+1)_{2n}}\cdots
   \underbrace{f(i+n)}_{i}\;\;
   \underbrace{f(i+n+1)}_{(i+1)_{2n}}\cdots\]\[\hspace{-6.3cm}
   \underbrace{f(i-x)}_{i}\;\;
   \underbrace{f(i-x+1)}_{(i+1)_{2n}}
\]
\end{itemize}
Hence, $i$ and $(i+1)_{2n}$ alternates in $f(w_i)$ for all $0\leq i\leq 2n-1$.

The alternation between $i$ and $(i-1)_{2n}$, $i$ and $(i+x)_{2n}$, $i$ and $(i-x)_{2n}$, as well as $i$ and $(i+n)_{2n}$ can be proved in a similar manner. Hence, in all cases, the required alterations hold. Thus, the word $f(w)$ represents the graph $G$ and consequently $R(G)\leq 5.$ For a detailed proof, we refer the reader to the appendix.

\end{proof}
\end{theorem}
\begin{corollary}\label{com}
    
    Let $G=C_{2n}(x,1,n)$ be a 5-regular circulant graph. Then $G$ is word-representable for $\frac{n}{2} \leq x \leq \frac{2n}{3}$ and $R(G) \leq 5.$
    \end{corollary}
 \begin{corollary}\label{cnxo}
 Let $G=C_{2n}(x,1,n)$ be a 5-regular circulant graph. If $n \equiv x \equiv 0 \pmod{3}$ and $\tfrac{n}{2}\leq x \leq n$, then $G$ is word-representable.      
\end{corollary}
\begin{proof}
    The result follows from Corollary \ref{com} and Theorem \ref{nx0}.
\end{proof}
 \begin{corollary}\label{cnox1}
 Let $G=C_{2n}(x,1,n)$ be a 5-regular circulant graph. If $n \equiv 0\pmod{3},  x \equiv 1 \pmod{3}$ and $1<x\leq\tfrac{2n}{3}$, then $G$ is word-representable.   
\end{corollary} 
\begin{proof}
    The result follows from Corollary \ref{com} and Theorem \ref{n0x1}.
\end{proof}

\begin{corollary}\label{cn2x1}
Let $G=C_{2n}(x,1,n)$ be a 5-regular circulant graph. If $n \equiv 2\pmod{3},  x \equiv 1 \pmod{3}$ and $1<x\leq \tfrac{2n}{3}$, then $G$ is word-representable.
\end{corollary}
\begin{proof}
    The result follows from Corollary \ref{com} and Theorem \ref{n2x1}.
\end{proof}
\begin{theorem}\label{12n}
    Let $G=C_{2n}(2,1,n)$ be a 5-regular circulant graph and $n>3$. Then $G$ is word-representable and $R(G) \leq 3.$ Furthermore, for $n=3$, $C_{6}(2,1,3)$ is word-representable and $R(G)=1.$
    \begin{proof}
    First, we consider the graph $G=C_{6}(2,1,3)$. Note that, $G\cong K_6 $. Hence, $G$ is word-representable and $R(G)=1.$ 
next, we consider the connected $5$-regular circulant graph $G = C(2,1,n)$, where $n>3$.
The vertex set of $G$ is 
\(
V(G) = \{0,1,2,\ldots,n-1\}.
\) 
By the definition of a circulant graph, each vertex $i \in V(G)$ is adjacent to each of
\(
(i+1)_{2n}, (i-1)_{2n}, (i+x)_{2n},  (i-x)_{2n}\) and \((i+n)_{2n}.
\) Define a morphism $f : V(G)^{\ast} \to V(G)^{\ast}$ by
\[
f(i) = (i)_{2n}(i-2)_{2n}(i+n)_{2n},
\forall i \in V(G).
\]
We shall show that the word \(f(w)\) represents \(G\), where
\(
w = 0\,1\,2\,\ldots\,(n-1).
\)
For each \(0 \leq i \leq n-1\), define
\(
w_i = i\,(i+1)\,\ldots\,(n-1)\,0\,1\,\ldots\,(i-1).
\)
By Proposition~2.6, the word \(f(w_i)\) represents \(G\) for all \(0 \leq i \leq n-1\). Consequently, two vertices \(i\) and \(j\) alternate in \(f(w)\) if and only if they alternate in \(f(w_i)\) for every \(i\).
From the definition of the morphism \(f\), each vertex \(i\) occurs exactly once in each of the words
$
f(i),\; f((i+2)_{2n})$ \text{and } $f((i+n)_{2n}).$
Hence, \(f(w)\) is a 3-uniform word. Moreover,
\[
f(w_i) = (i)\,(i-2)_{2n}\,(i+n)_{2n}\,(i+1)_{2n}\,(i-1)_{2n}\,(i+n+1)_{2n}\ldots\,(i-1)_{2n}\,(i-3)_{2n}\,(i+n-1)_{2n}.
\]

For any vertex \(i \in V(G)\), the presence of the factor \(f(i)\,f(i+1)\) in \(f(w_i)\) guarantees that \(i\) does not alternate with any vertex
\(
j \in V(G) \setminus \{(i+1)_{2n},\; (i-1)_{2n},\; (i+x)_{2n},\; (i-x)_{2n},\; (i+n)_{2n}\}.
\)

Therefore, it suffices to verify that \(i\) alternates with each vertex
\(
j \in \{(i+1)_{2n},\; (i-1)_{2n},\; (i+2)_{2n},\; (i-2)_{2n},\; (i+n)_{2n}\}
\)
within \(f(w_i)\), for all \(0 \leq i \leq n-1\).

\textbf{Case(1):} Here we will show the alternation between $i$ and $(i+1)_{2n}$ in $f(w_i)$.
Since $n\geq 3$, for every $0 \leq i \leq 2n-1$, the following ordering holds in $w_i$:  
\[
i \prec_{w_i}(i+1)_{2n}\prec_{w_i}(i+2)_{2n}\prec_{w_i}(i+3)_{2n}\prec_{w_i}(i+n)_{2n}\preceq_{w_i}(i+n+1)_{2n}.\]

By the definition of the morphism $f$, we have ,

    \[
f(w_i) 
= \underbrace{f(i)}_{i}\;\;
   \underbrace{f(i+1)}_{(i+1)_{2n}}\;\;
   \underbrace{f(i+2)}_{i}\;\;
   \underbrace{f(i+3)}_{(i+1)_{2n}}\cdots
   \underbrace{f(i+n)}_{i}
   \underbrace{f(i+n+1)}_{(i+1)_{2n}}.
\]
Hence, $i$ and $(i+1)_{2n}$ alternates in $f(w_i)$ for all $0\leq i\leq 2n-1$.

The alternation between $i$ and $(i-1)_{2n}$, $i$ and $(i+2)_{2n}$, $i$ and $(i-2)_{2n}$, as well as $i$ and $(i+n)_{2n}$ can be proved in a similar manner. Hence, in all cases, the required alterations hold. Thus, the word $f(w)$ represents the graph $G$ and consequently $R(G)\leq 3.$ For a detailed proof, we refer the reader to the appendix.

\end{proof}
\end{theorem}
\begin{corollary}\label{cnox2}
  Let $G=C_{2n}(x,1,n)$ be a 5-regular circulant graph. If $n \equiv 0\pmod{3},  x \equiv 2 \pmod{3}$ and $1<x\leq\tfrac{2n}{3}$, then $G$ is word-representable.    
\end{corollary}
\begin{proof}
    The result follows from Corollary \ref{com}, Theorem \ref{12n} and Theorem \ref{n0x2}.
\end{proof} 
\begin{corollary}\label{cn1x2}
Let $G=C_{2n}(x,1,n)$ be a 5-regular circulant graph. If $n \equiv 1\pmod{3},  x \equiv 2 \pmod{3}$ and $1<x\leq \tfrac{2n}{3}$, then $G$ is word-representable.
\end{corollary}
\begin{proof}
    The result follows from Corollary \ref{com}, Theorem \ref{12n} and Theorem \ref{n1x2}.
\end{proof}

\subsection{Word-representability of 5-Regular Circulant Graph Through Factorization}\label{factor}
In this section, we factorize \(5\)-regular circulant graphs into smaller regular circulant graphs. 
Using this factorization, we study their word-representability (Theorem \ref{6} and Theorem \ref{7}) and the representation number (Remark\ref{rem}). 
The factorization step primarily relies on Theorem~\ref{fac}.

\begin{theorem}\label{6}

Let $G=C_{2n}(a,b,n)$ be a 5-regular connected circulant graph, where $n,b$ are odd and $a$ is even. Then $C_{2n}(a,b,n)$ can be factorized as a Cartesian product of a 3-regular and a 2-regular circulant graph using Theorem \ref{fac} iff there exists positive integers $p(>2),q(1<q<n)$ such that $p\mid a$, $ q\mid b $, $pq=2n$ and $ gcd(p,q)=1$. 
\begin{proof}
$(\Rightarrow)$ Let $G=C_{2n}(a,b,n)$ be a 5-regular connected circulant graph, where $n,b$ are odd and $a$ is even. We assume that it can be factorized as a Cartesian product of a 3-regular and a 2-regular circulant graph using Theorem \ref{fac}. Then by Theorem \ref{fac} there exists $p(>1), q(>1)$ with $gcd(p,q)=1$ and $pq=2n$ such that $C_{pq}(T)= C_{p}(R) \square C_{q}(S) $. Here, $R \subseteq [1,\frac{p}{2}]$, $S \subseteq [1,\frac{q}{2}]$ and $T= qR \cup pS=\{a,b,n\}$.

If possible, let \(p = 2\) and \(q = n\). Then \(R = \{1\}\), i.e., \(qR = \{n\}\). 
Since \(p = 2\), all the elements of \(pS\) are even, which contradicts the fact that \(b \in pS\). 
Hence, we must have \(p > 2\) and \(q < n\).  

Now, since \(n\) is odd, it follows that \(p = 2n_{1}\) and \(q = n_{2}\), 
where both \(n_{1}\) and \(n_{2}\) are odd integers strictly greater than 1.  

Next, we claim that \(R = \{\tfrac{p}{2}, \tfrac{b}{q}\}\) and \(S = \{\tfrac{a}{p}\}\). 
As \(p\) is even, neither \(b\) nor \(n\) can belong to \(pS\). 
Thus, \(pS = \{a\}\) and \(qR = \{b, n\}\) and our claim is true.  

Since every element of \(R\) and \(S\) is an integer, it follows that \(p\) divides \(a\) 
and \(q\) divides \(b\).

$(\Leftarrow)$ Let us assume there exists positive integers $p(>2), q(1<q<n)$ such that $p\mid a$, $ q\mid b $, $pq=2n$ and $ gcd(p,q)=1$. We take $R=\{\frac{p}{2},\frac{b}{q}\}$ and $S=\{\frac{a}{p}\}$. Then, $T=qR \cup pS=\{a,b,n\}$. Also, we claim that, $R \subseteq [1,\frac{p}{2}]$, $S \subseteq [1,\frac{q}{2}]$. If possible, suppose \(\tfrac{b}{q} \geq \tfrac{p}{2}\) and \(\tfrac{a}{p} \geq \tfrac{q}{2}\), 
which is equivalent to \(b \geq n\) and \(a \geq n\). 
This contradicts the fact that \(a < n\) and \(b < n\). 
Hence, our claim holds. By Theorem \ref{fac}, \(G\) can be factorized 
as the Cartesian product of \(C_{p}(R)\) and \(C_{q}(S)\).

\end{proof}

\end{theorem}

\begin{theorem}\label{7}

Let $G=C_{2n}(a,b,n)$ be a 5-regular circulant graph, where $n,a$ are even and $b$ is odd. Then $C_{2n}(a,b,n)$ can be factorized as a Cartesian product of a 3-regular and a 2-regular circulant graph using Theorem \ref{fac} iff there exists positive integers $p(>2),q(1<q<n)$ such that $p\mid a$, $ q\mid b $, $pq=2n$ and $ gcd(p,q)=1$. 
\begin{proof}
$(\Rightarrow)$ Let $G=C_{2n}(a,b,n)$ be a 5-regular connected circulant graph, where $n,b$ are odd and $a$ is even. We assume that it can be factorized as a Cartesian product of a 3-regular and a 2-regular circulant graph using Theorem \ref{fac}. Then by Theorem \ref{fac} there exists $p(>1), q(>1)$ with $gcd(p,q)=1$ and $pq=2n$ such that $C_{pq}(T)= C_{p}(R) \square C_{q}(S) $. Here, $R \subseteq [1,\frac{p}{2}]$, $S \subseteq [1,\frac{q}{2}]$ and $T= qR \cup pS=\{a,b,n\}$.

If possible, let \(p = 2\) and \(q = n\). Then \(R = \{1\}\), i.e., \(qR = \{n\}\). 
Since \(p = 2\), all the elements of \(pS\) are even, which contradicts the fact that \(b \in pS\). 
Hence, we must have \(p > 2\) and \(q < n\).  

Now, since \(\gcd(p,q) = 1\), it follows that \(p = 2^{k}n_{1}\) and \(q = n_{2}\), 
where both \(n_{1}\) and \(n_{2} (>1)\) are odd integers.  

Next, we claim that 
$
R = \left\{\tfrac{p}{2}, \tfrac{b}{q}\right\},  S = \left\{\tfrac{a}{p}\right\}.
$

Since \(p\) is even, we have \(b \notin pS\). Now, if  \(n \in pS\) then \(\frac{q}{2}\) must be in $S$. But \(q=n_2\) is odd, hence \(\frac{q}{2}\) is not an integer. So, $n$ can not be in $pS.$ Therefore our claim is true.

Since every element of \(R\) and \(S\) is an integer, it follows that \(p\) divides \(a\) 
and \(q\) divides \(b\).

$(\Leftarrow)$ Let us assume there exists positive integers $p(>2), q(1<q<n)$ such that $p\mid a$, $ q\mid b $, $pq=2n$ and $ gcd(p,q)=1$. We take $R=\{\frac{p}{2},\frac{b}{q}\}$ and $S=\{\frac{a}{p}\}$. Then, $T=qR \cup pS=\{a,b,n\}$. Also, we claim that, $R \subseteq [1,\frac{p}{2}]$, $S \subseteq [1,\frac{q}{2}]$. If possible, suppose \(\tfrac{b}{q} \geq \tfrac{p}{2}\) and \(\tfrac{a}{p} \geq \tfrac{q}{2}\), 
which is equivalent to \(b \geq n\) and \(a \geq n\). 
This contradicts the fact that \(a < n\) and \(b < n\). 
Hence, our claim holds. By Theorem \ref{fac}, \(G\) can be factorized 
as the Cartesian product of \(C_{p}(R)\) and \(C_{q}(S)\).

\end{proof}

\end{theorem}
\begin{remark}\label{rem}
  It is well established that all \(3\)-regular (Theorem \ref{3reg}) and \(2\)-regular (Remark \ref{2reg}) circulant graphs are word-representable.  
Therefore, using Theorem \ref{carts}, we conclude that all the 5-regular circulant graphs in Theorems \ref{6} and \ref{7} are word-representable.  

Therefore, combining the above results, we deduce from Theorem \ref{rcarts} that
$
R(G) \leq 5 + \min\{p,q\}
$ for all graphs $G$ mentioned in Theorems \ref{6} and \ref{7}.

\end{remark}

\section{Conclusion and Future Work}\label{con}
We have studied 5-regular circulant graphs with respect to word-representability and established upper bounds for certain subclasses. We now present a conjecture and pose several open problems that provide possible ways to address it.
\begin{conjecture}
    All 5-regular circulant graphs are word-representable.
\end{conjecture}

\begin{problem}
    In Section~\ref{col}, we analyzed the word-representability of $C_{2n}(x,1,n)$ across several cases. A natural continuation of this work is to investigate the remaining configurations, such as $n \equiv 1 \pmod{3}$ with $x \equiv 0 \pmod{3}$, and $n \equiv 2 \pmod{3}$ with $x \equiv 0 \pmod{3}$, among others. Resolving these cases would yield a complete characterization of the class $C_{2n}(x,1,n)$. Furthermore, Upper bounds on the representation number can be obtained using morphisms, particularly for the intervals \(1 < x <\frac{n}{2}\) and \(\frac{2n}{3} < x < n\).

 \end{problem}
 \begin{problem}
     We have studied word-representability of 5-regular circulant graphs through Cartesian products in Section \ref{factor}. An alternative direction can be the factorization of this class via rooted products. For a detailed study of rooted products and their connection with word-representability, we refer the reader to \cite{srinivasan2024minimum}, \cite{kitaev2015words}, and \cite{broere2018word}.
 \end{problem}
\begin{problem}
    As observed from Theorem~\ref{thm:new}, setting \( n = 14 \) yields \( r = 3 \), which in turn 
provides an example of an \(8\)-regular circulant graph, \( C_{14}(3,4,5,6) \), that is not 
word-representable. A natural problem is to determine whether all $6$-regular and 
$7$-regular circulant graphs are word-representable, or whether these classes also contain counterexamples. 

\end{problem}
	\bibliographystyle{splncs04}
 \bibliography{mybibliography}
\section{Appendix}
\subsection{Proof of Theorem 25.}
  \textbf{Case(2):} Here we will show the alternation between $i$ and $(i-1)_{2n}$ in $f(w_i)$.  

Since $\tfrac{n}{2} < x \leq \tfrac{2n}{3}$ and $1<x<n$, we have
\(
(i+x)_{2n} \;\preceq_{w_i}\; (i+n-1)_{2n}\) and
\( 
(i+n)_{2n} \;\preceq_{w_i}\; (i-x-1)_{2n}.
\)
If possible let, $(i+n-1)_{2n} \;\prec_{w_i}\; (i+x)_{2n}$ and $(i-x-1)_{2n} \;\prec_{w_i}\; (i+n)_{2n}$, i.e, $i+n-1<i+x$ and $2n+i-x-1<i+n$. Both of these lead to $n<x+1$, which is a contradiction as $x<n$.  
Hence, for every $0 \leq i \leq 2n-1$, the following ordering holds in $w_i$:  
\[
i \prec_{w_i}(i+1)_{2n}\prec_{w_i}(i+x-1)_{2n}\prec_{w_i}(i+x)_{2n}\prec_{w_i}(i+n-1)_{2n}\preceq_{w_i}(i+n)_{2n}\prec_{w_i}(i-x-1)_{2n}\]\[\hspace{-8.5cm}\preceq_{w_i}(i-x)_{2n}\prec_{w_i}(i-1)_{2n}.\]

By the definition of the morphism $f$, we have four subcases,
\begin{itemize}
    \item $(i+x)_{2n}=(i+n-1)_{2n}$ and $(i+n)_{2n}=(i-x-1)_{2n}$
    \[
f(w_i) 
= \underbrace{f(i)}_{{i},(i-1)_{2n}}\;\;
   \underbrace{f(i+1)}_{i}\cdots
   \underbrace{f(i+x-1)}_{(i-1)_{2n}}\;\;
   \underbrace{f(i+x)}_{i,(i-1)_{2n}}\;\;
   \underbrace{f(i+n)}_{i,(i-1)_{2n},}\;\;
   \underbrace{f(i-x)}_{i}\cdots
   \underbrace{f(i-1)}_{(i+1)_{2n}}.
\]
    \item $(i+x)_{2n}\neq(i+n-1)_{2n}$ and $(i+n)_{2n}=(i-x-1)_{2n}$
    \[
f(w_i) 
= \underbrace{f(i)}_{i,(i-1)_{2n}}\;\;
   \underbrace{f(i+1)}_{i}\cdots
   \underbrace{f(i+x-1)}_{(i-1)_{2n}}\;\;
   \underbrace{f(i+x)}_{i}\cdots
   \underbrace{f(i+n-1)}_{(i-1)_{2n}}\;\;
   \underbrace{f(i+n)}_{i,(i-1)_{2n}}\;\;
   \underbrace{f(i-x)}_{i}\cdots\]\[\hspace{-7.8cm}
   \underbrace{f(i-1)}_{(i-1)_{2n}}.
\]
    \item $(i+x)_{2n}=(i+n-1)_{2n}$ and $(i+n)_{2n}\neq(i-x-1)_{2n}$
    \[
f(w_i) 
= \underbrace{f(i)}_{i,(i-1)_{2n}}
  \underbrace{f(i+1)}_{i}\cdots
   \underbrace{f(i+x-1)}_{(i-1)_{2n}}\;\;
   \underbrace{f(i+x)}_{i,(i-1)_{2n}}\;\;
   \underbrace{f(i+n)}_{i}\cdots
   \underbrace{f(i-x-1)}_{(i-1)_{2n}}\;\;
   \underbrace{f(i-x)}_{i}\cdots\]\[\hspace{-7.4cm}
   \underbrace{f(i-1)}_{(i-1)_{2n}}.
\]
     \item $(i+x)_{2n}\neq(i+n-1)_{2n}$ and $(i+n)_{2n}\neq(i-x-1)_{2n}$
     \[
f(w_i) 
= \underbrace{f(i)}_{i,(i-1)_{2n}}\;\;
   \underbrace{f(i+1)}_{i}\cdots
   \underbrace{f(i+x-1)}_{(i-1)_{2n}}\;\;
   \underbrace{f(i+x)}_{i}\cdots
   \underbrace{f(i+n-1)}_{(i-1)_{2n}}\;\;
   \underbrace{f(i+n)}_{i}\cdots
   \underbrace{f(i-x-1)}_{(i-1)_{2n}}\;\;\]
   \[\hspace{-6.3cm}\underbrace{f(i-x)}_{i}\cdots
   \underbrace{f(i-1)}_{(i-1)_{2n}}.
\]
\end{itemize}

Hence, $i$ and $(i-1)_{2n}$ alternates in $f(w_i)$ for all $0\leq i\leq 2n-1$.  

\textbf{Case(3):} Here we will show the alternation between $i$ and $(i-x)_{2n}$ in $f(w_i)$.    
Since $\tfrac{n}{2} < x \leq \tfrac{2n}{3}$ and $1<x<n$, we have
\( 
(i+1)_{2n} \;\preceq_{w_i}\; (i+n-x)_{2n} 
\),
\(
(i+n-x)_{2n} \;\prec_{w_i}\; (i+x)_{2n}\),
\(
(i+x)_{2n} \;\preceq_{w_i}\; (i-2x)_{2n}\),
\(
(i-2x)_{2n} \;\prec_{w_i}\; (i+n)_{2n}\)
and \(
(i+n)_{2n} \;\prec_{w_i}\; (i-x)_{2n}\). If possible let, $(i+x)_{2n} \;\preceq_{w_i}\; (i+n-x)_{2n}$ and $(i+n)_{2n} \;\preceq_{w_i}\; (i-2x)_{2n}$, i.e, $i+x\leq i+n-x$ and $i+n\leq 2n+i-2x$. Both of these lead to $2x\leq n$, which is a contradiction as $x>\tfrac{n}{2}$. If $(i+n-x)_{2n} \;\prec_{w_i}\; (i+1)_{2n}$ and $(i-x)_{2n} \;\preceq_{w_i}\; (i+n)_{2n}$, i.e, $i+n-x<i+1$ and $2n+i-x\leq i+n$. Both of these lead to a contradiction as $x<n$. Now, if $(i-2x)_{2n} \;\prec_{w_i}\; (i+x)_{2n}$ then $2n+i-2x<i+x$, i.e, $2n<3x$ which is clearly a contradiction as $x\leq \tfrac{2n}{3}$.
Hence, for every $0 \leq i \leq 2n-1$, the following ordering holds in $w_i$:  
\[
i \prec_{w_i}(i+1)_{2n}\preceq_{w_i}(i+n-x)_{2n}\prec_{w_i}(i+x)_{2n}\preceq_{w_i}(i-2x)_{2n}\prec_{w_i}(i+n)_{2n}\prec_{w_i}(i-x)_{2n}\prec_{w_i}(i\]\[\hspace{-11cm}-x+1)_{2n}.
\]

By the definition of the morphism $f$, we have four subcases,
\begin{itemize}
\item $(i+1)_{2n}=(i+n-x)_{2n}$ and $(i+x)_{2n}=(i-2x)_{2n}$
   \[f(w_i) 
= \underbrace{f(i)}_{{i},(i-x)_{2n}}\;\;
   \underbrace{f(i+1)}_{i,(i-x)_{2n}}\cdots
   \underbrace{f(i+x)}_{i,(i-x)_{2n}}\cdots
   \underbrace{f(i+n)}_{i}\cdots
   \underbrace{f(i-x)}_{(i-x)_{2n},i}\;\;
   \underbrace{f(i-x+1)}_{(i-x)_{2n}}.
\]
\item $(i+1)_{2n}\neq(i+n-x)_{2n}$ and $(i+x)_{2n}=(i-2x)_{2n}$
    \[
f(w_i) 
= \underbrace{f(i)}_{{i},(i-x)_{2n}}\;\;
   \underbrace{f(i+1)}_{i}\cdots
   \underbrace{f(i+n-x)}_{(i-x)_{2n}}\cdots
   \underbrace{f(i+x)}_{i,(i-x)_{2n}}\cdots
   \underbrace{f(i+n)}_{i}\cdots
   \underbrace{f(i-x)}_{(i-x)_{2n},i}\;\;
   \underbrace{f(i-x+1)}_{(i-x)_{2n}}
\]
    \item $(i+1)_{2n}=(i+n-x)_{2n}$ and $(i+x)_{2n}\neq(i-2x)_{2n}$
    \[
f(w_i) 
= \underbrace{f(i)}_{{i},(i-x)_{2n}}\;\;
   \underbrace{f(i+1)}_{i,(i-x)_{2n}}\cdots
   \underbrace{f(i+x)}_{i}\cdots
   \underbrace{f(i-2x)}_{(i-x)_{2n}}\cdots
   \underbrace{f(i+n)}_{i}\cdots
   \underbrace{f(i-x)}_{(i-x)_{2n},i}\;\;
   \underbrace{f(i-x+1)}_{(i-x)_{2n}}
\]
    \item $(i+1)_{2n}\neq(i+n-x)_{2n}$ and $(i+x)_{2n}\neq(i-2x)_{2n}$
     \[
f(w_i) 
= \underbrace{f(i)}_{{i},(i-x)_{2n}}\;\;
   \underbrace{f(i+1)}_{i}\cdots
   \underbrace{f(i+n-x)}_{(i-x)_{2n}}\cdots
   \underbrace{f(i+x)}_{i}\cdots
   \underbrace{f(i-2x)}_{(i-x)_{2n}}\cdots
   \underbrace{f(i+n)}_{i}\cdots
   \underbrace{f(i-x)}_{(i-x)_{2n},i}\;\;\]\[\hspace{-7.5cm}
   \underbrace{f(i-x+1)}_{(i-x)_{2n}}
\]
\end{itemize}
Hence, $i$ and $(i-x)_{2n}$ alternates in $f(w_i)$ for all $0\leq i\leq 2n-1$.  

\textbf{Case(4):}
Here we will show the alternation between $i$ and $(i+x)_{2n}$ in $f(w_i)$.    
Since $\tfrac{n}{2} < x \leq \tfrac{2n}{3}$ and $1<x<n$, we have
\( 
(i+x+1)_{2n} \;\preceq_{w_i}\; (i+n)_{2n} 
\),
\(
(i+n)_{2n} \;\prec_{w_i}\; (i+2x)_{2n}\),
\(
(i+2x)_{2n} \;\preceq_{w_i}\; (i-x)_{2n}\) and
\(
(i-x)_{2n} \;\prec_{w_i}\; (i+n+x)_{2n}\). If possible let, $(i+2x)_{2n} \;\preceq_{w_i}\; (i+n)_{2n}$ and $(i+n+x)_{2n} \;\preceq_{w_i}\; (i-x)_{2n}$,i.e, $i+2x\leq i+n$ and $i+n+x\leq 2n+i-x$. Both of these lead to $2x\leq n$, which is a contradiction as $x>\tfrac{n}{2}$. If $(i+n)_{2n} \;\prec_{w_i}\; (i+x+1)_{2n}$, i.e, $i+n<i+x+1$.This leads to a contradiction as $x<n$. Now, if $(i-x)_{2n} \;\prec_{w_i}\; (i+2x)_{2n}$ then $2n+i-x<i+2x$, i.e, $2n<3x$ which is clearly a contradiction as $x\leq \tfrac{2n}{3}$.
Hence, for every $0 \leq i \leq 2n-1$, the following ordering holds in $w_i$:  
\[
i \prec_{w_i}(i+1)_{2n}\prec_{w_i}(i+x)_{2n}\prec_{w_i}(i+x+1)_{2n}\preceq_{w_i}(i+n)_{2n}\prec_{w_i}(i+2x)_{2n}\preceq_{w_i}(i-x)_{2n}\prec_{w_i}(i\]\[\hspace{-11cm}+n+x)_{2n}.
\]

By the definition of the morphism $f$, we have four subcases,
\begin{itemize}
  \item $(i+x+1)_{2n}=(i+n)_{2n}$ and $(i+2x)_{2n}=(i-x)_{2n}$
    \[
f(w_i) 
= \underbrace{f(i)}_{{i},(i+x)_{2n}}\;\;
   \underbrace{f(i+1)}_{i}\cdots
   \underbrace{f(i+x)}_{(i+x)_{2n},i}\;\;
   \underbrace{f(i+x+1)}_{(i+x)_{2n},i}\cdots
   \underbrace{f(i+2x)}_{(i+x)_{2n},i}\cdots
   \underbrace{f(i+n+x)}_{(i+x)_{2n}}.
\]
    \item $(i+x+1)_{2n}\neq(i+n)_{2n}$ and $(i+2x)_{2n}=(i-x)_{2n}$
    \[
f(w_i) 
= \underbrace{f(i)}_{{i},(i+x)_{2n}}\;\;
   \underbrace{f(i+1)}_{i}\cdots
   \underbrace{f(i+x)}_{(i+x)_{2n},i}\;\;
   \underbrace{f(i+x+1)}_{(i+x)_{2n}}\cdots
   \underbrace{f(i+n)}_{i}\cdots
   \underbrace{f(i+2x)}_{(i+x)_{2n},i}\cdots
   \underbrace{f(i+n+x)}_{(i+x)_{2n}}.
\]
    \item $(i+x+1)_{2n}=(i+n)_{2n}$ and $(i+2x)_{2n}\neq(i-x)_{2n}$
    \[
f(w_i) 
= \underbrace{f(i)}_{{i},(i+x)_{2n}}\;\;
   \underbrace{f(i+1)}_{i}\cdots
   \underbrace{f(i+x)}_{(i+x)_{2n},i}\;\;
   \underbrace{f(i+x+1)}_{(i+x)_{2n},i}\cdots
   \underbrace{f(i+2x)}_{(i+x)_{2n}}\cdots
   \underbrace{f(i-x)}_{i}\cdots
   \underbrace{f(i+n+x)}_{(i+x)_{2n}}.
\]
     \item $(i+x+1)_{2n}\neq(i+n)_{2n}$ and $(i+2x)_{2n}\neq(i-x)_{2n}$
     \[
f(w_i) 
= \underbrace{f(i)}_{{i},(i+x)_{2n}}\;\;
   \underbrace{f(i+1)}_{i}\cdots
   \underbrace{f(i+x)}_{(i+x)_{2n},i}\;\;
   \underbrace{f(i+x+1)}_{(i+x)_{2n}}\cdots
   \underbrace{f(i+n)}_{i}\cdots
   \underbrace{f(i+2x)}_{(i+x)_{2n}}\cdots
   \underbrace{f(i-x)}_{i}\cdots\]
   \[\hspace{-7.2cm}\underbrace{f(i+n+x)}_{(i+x)_{2n}}.
\]
\end{itemize}
Hence, $i$ and $(i+x)_{2n}$ alternates in $f(w_i)$ for all $0\leq i\leq 2n-1$.  

\textbf{Case(5):} Here we will show the alternation between $i$ and $(i+n)_{2n}$ in $f(w_i)$.    
Since $\tfrac{n}{2} < x \leq \tfrac{2n}{3}$ and $1<x<n$, we have
\( 
(i+1)_{2n} \;\preceq_{w_i}\; (i+n-x)_{2n} 
\),
\(
(i+n-x)_{2n} \;\prec_{w_i}\; (i+x)_{2n}\),
\(
(i+n+1)_{2n} \;\preceq_{w_i}\; (i-x)_{2n}\) and
\(
(i-x)_{2n} \;\prec_{w_i}\; (i+n+x)_{2n}\). If possible let, $(i+x)_{2n} \;\preceq_{w_i}\; (i+n-x)_{2n}$ and $(i+n+x)_{2n} \;\preceq_{w_i}\; (i-x)_{2n}$, i.e, $i+x\leq i+n-x$ and $i+n+x\leq 2n+i-x$. Both of these lead to $2x\leq n$, which is a contradiction as $x>\tfrac{n}{2}$. If $(i+n-x)_{2n} \;\prec_{w_i}\; (i+1)_{2n}$ and $(i-x)_{2n} \;\prec_{w_i}\; (i+n+1)_{2n}$, i.e, $i+n-x<i+1$ and $2n+i-x< i+n+1$. Both of these lead to a contradiction as $x<n$.
Hence, for every $0 \leq i \leq 2n-1$, the following ordering holds in $w_i$:  
\[
i \prec_{w_i}(i+1)_{2n}\preceq_{w_i}(i+n-x)_{2n}\prec_{w_i}(i+x)_{2n}\prec_{w_i}(i+n)_{2n}\prec_{w_i}(i+n+1)_{2n}\preceq_{w_i}(i-x)_{2n}\]\[\hspace{-10cm}\prec_{w_i}(i+n+x)_{2n}.
\]

By the definition of the morphism $f$, we have four subcases,
\begin{itemize}
  \item $(i+1)_{2n}=(i+n-x)_{2n}$ and $(i+n+1)_{2n}=(i-x)_{2n}$
    \[
f(w_i) 
= \underbrace{f(i)}_{{i},(i+n)_{2n}}\;\;
   \underbrace{f(i+1)}_{i,(i+n)_{2n}}\cdots
   \underbrace{f(i+x)}_{i}\cdots
   \underbrace{f(i+n)}_{(i+n)_{2n},i}\;\;
   \underbrace{f(i+n+1)}_{(i+n)_{2n},i}\cdots
   \underbrace{f(i+n+x)}_{(i+n)_{2n}}.
\]
    \item $(i+1)_{2n}\neq(i+n-x)_{2n}$ and $(i+n+1)_{2n}=(i-x)_{2n}$
    \[
f(w_i) 
= \underbrace{f(i)}_{{i},(i+n)_{2n}}\;\;
   \underbrace{f(i+1)}_{i}\cdots
   \underbrace{f(i+n-x)}_{(i+n)_{2n}}\cdots
   \underbrace{f(i+x)}_{i}\cdots
   \underbrace{f(i+n)}_{(i+n)_{2n},i}\;\;
   \underbrace{f(i+n+1)}_{(i+n)_{2n},i}\cdots\cdots\]\[\hspace{-6.7cm}
   \underbrace{f(i+n+x)}_{(i+n)_{2n}}.
\]
    \item $(i+1)_{2n}=(i+n-x)_{2n}$ and $(i+n+1)_{2n}\neq(i-x)_{2n}$
    \[
f(w_i) 
= \underbrace{f(i)}_{{i},(i+n)_{2n}}\;\;
   \underbrace{f(i+1)}_{i,(i+n)_{2n}}\cdots
   \underbrace{f(i+x)}_{i}\cdots
   \underbrace{f(i+n)}_{(i+n)_{2n},i}\;\;
   \underbrace{f(i+n+1)}_{(i+n)_{2n}}\cdots
   \underbrace{f(i-x)}_{i}\cdots
   \underbrace{f(i+n+x)}_{(i+n)_{2n}}.
\]
    \item $(i+1)_{2n}\neq(i+n-x)_{2n}$ and $(i+n+1)_{2n}\neq(i-x)_{2n}$
     \[
f(w_i) 
= \underbrace{f(i)}_{{i},(i+n)_{2n}}\;\;
   \underbrace{f(i+1)}_{i}\cdots
   \underbrace{f(i+n-x)}_{(i+n)_{2n}}\cdots
   \underbrace{f(i+x)}_{i}\cdots
   \underbrace{f(i+n)}_{(i+n)_{2n},i}\;\;
   \underbrace{f(i+n+1)}_{(i+n)_{2n}}\cdots
   \underbrace{f(i-x)}_{i}\]
   \[\hspace{-6.9cm}\cdots\underbrace{f(i+n+x)}_{(i+n)_{2n}}.
\]
\end{itemize}
Hence, $i$ and $(i+n)_{2n}$ alternates in $f(w_i)$ for all $0\leq i\leq 2n-1$.
\subsection{Proof of Theorem 26.}

\textbf{Case(2):}  Here we will show the alternation between $i$ and $(i-1)_{2n}$ in $f(w_i)$.  
Since $x=\frac{n}{2}$ and $x\geq2$ and $n\geq4$, we have
\(
(i+x)_{2n} \;\prec_ {w_i}\; (i+n-1)_{2n}\) and \((i+n)_{2n} \;\prec_ {w_i}\; (i-x-1)_{2n}.
\)
If possible let, $(i+n-1)_{2n} \;\preceq_{w_i}\; (i+x)_{2n}$ and $(i-x-1)_{2n} \;\preceq_{w_i}\; (i+n)_{2n}$, i.e, $i+n-1\leq i+x$ and $2n+i-x-1\leq i+n$. Both of these lead to $n\leq x+1$, i.e, $n\leq 2$, which is a contradiction.
Hence, for every $0 \leq i \leq 2n-1$, the following ordering holds in $w_i$:  
\[ i \prec_{w_i}(i+1) \prec_{w_i}(i+x-1)_{2n}\prec_{w_i}(i+x)_{2n}\prec_{w_i}(i+n-1)_{2n}\prec_{w_i}(i+n)_{2n}\prec_{w_i}(i-x-1)_{2n}\]\[\hspace{-10.5cm}\prec_{w_i}(i-x)_{2n}.\]
By the definition of the morphism $f$, we have
    \[
f(w_i) 
= \underbrace{f(i)}_{i,( i-1)_{2n}}\;\;
   \underbrace{f(i+1)}_{i}\cdots
   \underbrace{f(i+x-1)}_{(i-1)_{2n}}\;\;
   \underbrace{f(i+x)}_{i}\cdots
   \underbrace{f(i+n-1)}_{ (i-1)_{2n }}\;\;
   \underbrace{f(i+n)}_{i}\cdots
   \underbrace{f(i-x-1)}_{ (i-1)_{2n} }\;\;\]\[\hspace{-7cm}
   \underbrace{f(i-x)}_{i}\cdots
   \underbrace{f(i-1)}_{(i-1)_{2n}}
\]
Hence, $i$ and $(i-1)_{2n}$ alternates in $f(w_i)$ for all $0\leq i\leq 2n-1$.

\textbf{Case(3):} Here we will show the alternation between $i$ and $(i-x)_{2n}$ in $f(w_i)$.
Since $x=\frac{n}{2}$ and $x\geq2$ and $n\geq4$, we have
\(
(i+n)_{2n} \;\prec_ {w_i}\; (i-x)_{2n}\) 
If possible let, $(i-x)_{2n} \;\preceq_{w_i}\; (i+n)_{2n}$, i.e, $2n+i-x\leq i+n$. This leads to $n\leq x$, which is a contradiction.
Hence, for every $0 \leq i \leq 2n-1$, the following ordering holds in $w_i$:  
\[ i \prec_{w_i}(i+1) \prec_{w_i}(i+x)_{2n}\prec_{w_i}(i+n)_{2n}\prec_{w_i}(i-x)_{2n}\prec_{w_i}(i-x+1)_{2n}.\]
By the definition of the morphism $f$, we have
    \[
f(w_i) 
= \underbrace{f(i)}_{i,( i-x)_{2n}}\;\;
   \underbrace{f(i+1)}_{i}\cdots
   \underbrace{f(i+x)}_{(i-x)_{2n},i}\cdots
   \underbrace{f(i+n)}_{(i-x)_{2n},i }\cdots
   \underbrace{f(i-x)}_{ (i-x)_{2n},i}\;\;
   \underbrace{f(i-x+1)}_{ (i-x)_{2n} }
\]
Hence, $i$ and $(i-x)_{2n}$ alternates in $f(w_i)$ for all $0\leq i\leq 2n-1$.  

\textbf{Case(4):} Here we will show the alternation between $i$ and $(i+x)_{2n}$ in $f(w_i)$.
Since $x=\frac{n}{2}$ and $x\geq2$ and $n\geq4$, we have
\(
(i+x+1)_{2n} \;\prec_ {w_i}\; (i+n)_{2n}\) and \(
(i+n)_{2n} \;\prec_ {w_i}\; (i-x)_{2n}\) 
If possible let, $(i+n)_{2n} \;\preceq_{w_i}\; (i+x+1)_{2n}$ and $(i-x)_{2n} \;\preceq_{w_i}\; (i+n)_{2n}$, i.e, $i+n \leq i+x+1$ and $2n+i-x\leq i+n$. The first inequality leads to $n\leq x+1$, i.e, $n\leq 2$, which is a contradiction. Similarly, the second one leads to $n\leq x$, which is also a contradiction.
Hence, for every $0 \leq i \leq 2n-1$, the following ordering holds in $w_i$:  
\[ i \prec_{w_i}(i+1) \prec_{w_i}(i+x)_{2n}\prec_{w_i}(i+x+1)_{2n}\prec_{w_i}(i+n)_{2n}\prec_{w_i}(i-x)_{2n}.\]
By the definition of the morphism $f$, we have
    \[
f(w_i) 
= \underbrace{f(i)}_{i,( i+x)_{2n}}\;\;
   \underbrace{f(i+1)}_{i}\cdots
   \underbrace{f(i+x)}_{(i+x)_{2n},i}\;\;
   \underbrace{f(i+x+1)}_{(i+x)_{2n}}\cdots
   \underbrace{f(i+n)}_{ i,(i+x)_{2n}}\cdots
   \underbrace{f(i-x)}_{ i,(i+x)_{2n} }
\]
Hence, $i$ and $(i+x)_{2n}$ alternates in $f(w_i)$ for all $0\leq i\leq 2n-1$.

\textbf{Case(5):} Here we will show the alternation between $i$ and $(i+n)_{2n}$ in $f(w_i)$.    
Since $ x=\frac{n}{2}$ and $x\geq2$ and $n\geq4$, we have
\(
(i+n+1)_{2n} \;\prec_ {w_i}\; (i-x)_{2n}\) 

If possible let, $(i-x)_{2n} \;\preceq_{w_i}\; (i+n+1)_{2n}$, i.e, $2n+i-x\leq i+n+1$. This leads to $n\leq x+1$, i.e, $n\leq 2$, which is a contradiction.
Hence, for every $0 \leq i \leq 2n-1$, the following ordering holds in $w_i$:  
\[ i \prec_{w_i}(i+1) \prec_{w_i}(i+x)_{2n}\prec_{w_i}(i+n)_{2n}\prec_{w_i}(i+n+1)_{2n}\prec_{w_i}(i-x)_{2n}.\]
By the definition of the morphism $f$, we have
    \[
f(w_i) 
= \underbrace{f(i)}_{i,( i+n)_{2n}}\;\;
   \underbrace{f(i+1)}_{i}\cdots
   \underbrace{f(i+x)}_{(i+n)_{2n},i}\cdots
   \underbrace{f(i+n)}_{ (i+n)_{2n},i}\;\;
\underbrace{f(i+n+1)}_{(i+n)_{2n}}\cdots
   \underbrace{f(i-x)}_{ i,(i+n)_{2n} }
\]
Hence, $i$ and $(i+n)_{2n}$ alternates in $f(w_i)$ for all $0\leq i\leq 2n-1$.

\subsection{Proof of Theorem 27.}
\textbf{Case(2):} Here we will show the alternation between $i$ and $(i-1)_{2n}$ in $f(w_i)$.
Since $n> 3$, for every $0 \leq i \leq 2n-1$, the following ordering holds in $w_i$:  
\[
i \prec_{w_i}(i+1)_{2n}\prec_{w_i}(i+2)_{2n}\prec_{w_i}(i+n-1)_{2n}\prec_{w_i}(i+n)_{2n}\preceq_{w_i}(i-1)_{2n}.\]

By the definition of the morphism $f$, we have,

    \[
f(w_i) 
= \underbrace{f(i)}_{i}\;\;
   \underbrace{f(i+1)}_{(i-1)_{2n}}\;\;
   \underbrace{f(i+2)}_{i}\cdots
   \underbrace{f(i+n-1)}_{(i-1)_{2n}}\;\;
   \underbrace{f(i+n)}_{i}
   \underbrace{f(i-1)}_{(i-1)_{2n}}.
\]
Hence, $i$ and $(i-1)_{2n}$ alternates in $f(w_i)$ for all $0\leq i\leq 2n-1$.

\textbf{Case(3):} Here we will show the alternation between $i$ and $(i+2)_{2n}$ in $f(w_i)$.
Since $n> 3$, for every $0 \leq i \leq 2n-1$, the following ordering holds in $w_i$:  
\[
i \prec_{w_i}(i+2)_{2n}\prec_{w_i}(i+4)_{2n}\preceq_{w_i}(i+n)_{2n}\prec_{w_i}(i+n+2)_{2n}.\]

By the definition of the morphism $f$, we have two subcases,
\begin{itemize}
    \item $(i+4)_{2n}=(i+n)_{2n}$
    \[
f(w_i) 
= \underbrace{f(i)}_{i}\cdots
   \underbrace{f(i+2)}_{(i+2)_{2n},i}\cdots
   \underbrace{f(i+4)}_{(i+2)_{2n},i}\cdots
   \underbrace{f(i+n+2)}_{(i+2)_{2n}}.
\]
     \item $(i+4)_{2n}\neq(i+n)_{2n}$
    \[
f(w_i) 
= \underbrace{f(i)}_{i}\cdots
   \underbrace{f(i+2)}_{(i+2)_{2n},i}\cdots
   \underbrace{f(i+4)}_{(i+2)_{2n}}\cdots
   \underbrace{f(i+n)}_{i}\cdots
   \underbrace{f(i+n+2)}_{(i+2)_{2n}}.
\]
\end{itemize}
Hence, $i$ and $(i+2)_{2n}$ alternates in $f(w_i)$ for all $0\leq i\leq 2n-1$.

\textbf{Case(4):} Here we will show the alternation between $i$ and $(i-2)_{2n}$ in $f(w_i)$.
Since $n>3$, for every $0 \leq i \leq 2n-1$, the following ordering holds in $w_i$:  
\[
i \prec_{w_i}(i+2)_{2n}\preceq_{w_i}(i+n-2)_{2n}\prec_{w_i}(i+n)_{2n}\prec_{w_i}(i-2)_{2n}.\]

By the definition of the morphism $f$, we have two subcases,
\begin{itemize}
    \item $(i+2)_{2n}=(i+n-2)_{2n}$
    \[
f(w_i) 
= \underbrace{f(i)}_{i,(i-2)_{2n}}\cdots
   \underbrace{f(i+2)}_{i,(i-2)_{2n}}\cdots
   \underbrace{f(i+n)}_{i}\cdots
   \underbrace{f(i-2)}_{(i-2)_{2n}}.
\]
     \item $(i+2)_{2n}\neq(i+n-2)_{2n}$
    \[
f(w_i) 
= \underbrace{f(i)}_{i,(i-2)_{2n}}\cdots
   \underbrace{f(i+2)}_{i}\cdots
   \underbrace{f(i+n-2)}_{(i-2)_{2n}}\cdots
   \underbrace{f(i+n)}_{i}\cdots
   \underbrace{f(i-2)}_{(i-2)_{2n}}.
\]
\end{itemize}
Hence, $i$ and $(i-2)_{2n}$ alternates in $f(w_i)$ for all $0\leq i\leq 2n-1$.

\textbf{Case(5):} Here we will show the alternation between $i$ and $(i+n)_{2n}$ in $f(w_i)$.
Since $n> 3$, for every $0 \leq i \leq 2n-1$, the following ordering holds in $w_i$:  
\[
i \prec_{w_i}(i+2)_{2n}\prec_{w_i}(i+n)_{2n}\prec_{w_i}(i+n+2)_{2n}.\]

By the definition of the morphism $f$, we have,

    \[
f(w_i) 
= \underbrace{f(i)}_{i,(i+n)_{2n}}\;\;
   \underbrace{f(i+2)}_{i}\cdots
   \underbrace{f(i+n)}_{(i+n)_{2n},i}\cdots
   \underbrace{f(i+n+2)}_{(i+n)_{2n}}.
\]
Hence, $i$ and $(i+n)_{2n}$ alternates in $f(w_i)$ for all $0\leq i\leq 2n-1$.
\end{document}